\documentclass[10pt,reqno]{article}
\usepackage[b5paper,top=1.2in, left=0.9in]{geometry}
\usepackage{verbatim}
\usepackage{amssymb}
\usepackage{amsmath}
\usepackage{graphicx}
\usepackage{appendix}
\usepackage{color}
\usepackage{amsthm}
\usepackage{tikz}
\usepackage{ifthen}
\usepackage{float}
\usepackage{subfloat}
\usepackage{caption}
\usepackage{subcaption}
\usetikzlibrary{arrows,positioning,decorations.pathmorphing,  decorations.markings}

\renewcommand{\d}{\mathrm{d}}
\newcommand{\D}{\mathrm{D}}

\newtheorem{Thm}{Theorem}[section]
\newtheorem{Lem}[Thm]{Lemma}
\newtheorem{Prop}[Thm]{Proposition}
\newtheorem{Cor}[Thm]{Corollary}
\newtheorem{Rem}[Thm]{Remark}
\newtheorem{Def}[Thm]{Definition}

\newtheorem{Ex}[Thm]{Example}

\newtheorem{Nota}[Thm]{Notation}

\newtheoremstyle{named}{}{}{\itshape}{}{\bfseries}{.}{.5em}{#1 #3}
\theoremstyle{named}

\def\R{\mathbb{R}}
\def\Q{\mathbb{Q}}

\def\C{\mathbb{C}}
\def\Z{\mathbb{Z}}

\def\fb{\mathfrak{b}}
\def\g{\mathfrak{g}}

\def\sl{\mathfrak{sl}}

\def\cD{\mathcal{D}}
\def\cE{\mathcal{E}}
\def\cF{\mathcal{F}}

\def\cH{\mathcal{H}}
\def\cI{\mathcal{I}}

\def\cM{\mathcal{M}}

\def\cP{\mathcal{P}}

\def\cR{\mathcal{R}}
\def\cS{\mathcal{S}}

\def\cU{\mathcal{U}}

\def\cX{\mathcal{X}}

\def\a{\alpha}
\def\b{\beta}
\def\c{\gamma}

\def\D{\Delta}

\def\d{\delta}

\def\l{\lambda}
\def\L{\Lambda}

\def\s{\sigma}

\def\w{\omega}

\def\bd{\mathbf{d}}

\def\be{\mathbf{e}}

\def\bf{\mathbf{f}}

\def\bi{\mathbf{i}}

\def\bo{\mathbf{o}}

\def\bT{\mathbf{T}}

\def\=>{\Longrightarrow}
\def\inj{\hookrightarrow}
\def\corr{\longleftrightarrow}

\def\to{\longrightarrow}

\def\ox{\otimes}
\def\o+{\oplus}
\def\bo+{\bigoplus}

\def\<{\langle}
\def\>{\rangle}
\def\({\left(}
\def\){\right)}
\def\oo{\infty}

\def\^{\wedge}
\def\+{\dagger}

\def\sub{\subset}
\def\inv{^{-1}}
\def\half{\frac{1}{2}}

\def\dd[#1,#2]{\frac{d#1}{d#2}}
\def\del[#1,#2]{\frac{\partial #1}{\partial #2}}
\def\over[#1]{\overline{#1}}
\def\vec[#1]{\overrightarrow{#1}}

\def\tab{\;\;\;\;\;\;}

\newcommand{\til}[1]{\widetilde{#1}}
\newcommand{\what}[1]{\widehat{#1}}

\newcommand{\case}[2][lllllllllllllllllllllllllllllllllllll]{\left\{\begin{array}{#1}#2 \\ \end{array}\right.}
\newcommand{\Eq}[1]{\begin{align}#1\end{align}}
\newcommand{\Eqn}[1]{\begin{align*}#1\end{align*}}
\tikzset{>=latex}
\tikzstyle{vthick}=[line width=1.8pt]
\newcommand*{\DashedArrow}[1][]{\mathbin{\tikz [baseline=-0.25ex,-latex, dashed,#1] \draw [#1] (0pt,0.5ex) -- (1.3em,0.5ex);}}%

\newcommand\drawpath[2]{%
  \foreach \too [count=\c from 1] in {#1}
  {
  \ifthenelse{\c=1}
  {\xdef\from{\too}}
  {\path (\from) edge [->, #2] (\too);
    \xdef\from{\too}}
  };
}

\begin{document}
\title{Cluster realization of positive representations of \\split real quantum Borel subalgebra}

\author{  Ivan Chi-Ho Ip\footnote{
         	   Center for the Promotion of Interdisciplinary Education and Research/\newline
		   Department of Mathematics, Graduate School of Science, Kyoto University, Japan
		\newline
		Email: ivan.ip@math.kyoto-u.ac.jp
          }
}
         
\date{\today}

\numberwithin{equation}{section}

\maketitle
\begin{center}
\vspace{-5ex}
\textit{Dedicated to the memory of Ludvig D. Faddeev}
\end{center}

\begin{abstract}
In our previous work \cite{Ip14}, we studied the positive representations of split real quantum groups $\mathcal{U}_{q\tilde{q}}(\mathfrak{g}_\mathbb{R})$ restricted to its Borel part, and showed that they are closed under taking tensor products. However, the tensor product decomposition was only constructed abstractly using the GNS-representation of a $C^*$-algebraic version of the Drinfeld-Jimbo quantum groups. In this paper, using the recently discovered cluster realization of quantum groups \cite{Ip16}, we write down the decomposition explicitly by realizing it as a sequence of cluster mutations in the corresponding quiver diagram representing the tensor product.
\end{abstract}

{\small  {\textbf {2010 Mathematics Subject Classification.} Primary 81R50, Secondary 22D25}}

{\small  {\textbf{Keywords.} positive representations, split real quantum groups, modular double, quantum cluster algebra, tensor category}
\section{Introduction}\label{sec:intro}
\subsection{Positive representations of split real quantum groups}
To any finite dimensional complex simple Lie algebra $\g$, Drinfeld \cite{D} and Jimbo \cite{J} defined a remarkable Hopf algebra $\cU_q(\g)$ known as the quantum group. The notion of \emph{positive representations} was introduced in \cite{FI} as a new research program devoted to the representation theory of its split real form $\cU_{q\til{q}}(\g_\R)$ which uses the concept of Faddeev's modular double \cite{Fa1,Fa2}, and generalizes the case of $\cU_{q\til{q}}(\sl(2,\R))$ studied extensively by Teschner {\it et al.} \cite{BT, PT1, PT2} from the physics point of view.

Explicit construction of the positive representations $\cP_\l$ of $\cU_{q\til{q}}(\g_\R)$, parametrized by the $\R_+$-span of positive weights $\l\in P_\R^+$, was constructed in \cite{Ip2,Ip3} for simple Lie algebra $\g$ of all types, where the generators of the quantum groups are realized by positive essentially self-adjoint operators acting on certain Hilbert space, and compatible with the modular double structure (see Section \ref{sec:prelim:pos} for a review). Although the representations involve unbounded operators, the algebraic relations are well-defined and are unitary equivalent to the \emph{integrable representations} of the quantum plane in the sense of Schm\"udgen \cite{Ip1, Sch}.

A long standing conjecture in the theory of positive representations is the closure under taking tensor product:
\Eq{
\cP_\a\ox \cP_\b\simeq \int_{\c\in P_\R^+} \cP_\c \ox \cM_{\a\b}^{\c} d\mu(\c)
}
for some Plancherel measure $d\mu(\c)$ and multiplicity module $\cM_{\a\b}^{\c}$. The simplest case of $\cU_{q\til{q}}(\sl(2,\R))$ has been proved previously in \cite{PT2}, which is related to the fusion relations of quantum Liouville theory. The case in type $A_n$ is recently solved algebraically in \cite{SS17} using a cluster realization \cite{Ip16,SS16} of the positive representations discussed below, and is related to certain quantum open Toda systems. For other types however, the question is still open, but as a first step, we showed in \cite{Ip14} that the restriction of $\cP_\l$ to the Borel part $\cU_{q\til{q}}(\fb_\R)$ is indeed closed under taking tensor product:
\begin{Thm}\label{mainthm} Let $\cP_\l^\fb$ be the positive representations $\cP_\l$ restricted to the Borel part $\cU_{q\til{q}}(\fb_\R)$. Then $\cP_\l^\fb\simeq \cP^\fb$ does not depend on $\l$, and we have the unitary equivalence
\Eq{\label{decomp}
\cP^\fb\ox \cP^\fb \simeq \cP^\fb\ox \cM,
}
where $\cM$ is a multiplicity module in which $\cU_{q\til{q}}(\fb_\R)$ acts trivially.
\end{Thm}
As applications, this gives a new candidate for quantum higher Teichm\"uller theory, where the above unitary equivalence gives the \emph{quantum mutation operator} satisfying the pentagon relation (see \cite{FK, Ip14} for a review). It also provides a major step towards proving the tensor product decomposition of $\cP_\l$ in general. Together with the braiding by the universal $\cR$ operator, the positive representations will carry a (continuous) braided tensor category structure, which may give rise to new class of TQFT's in the sense of Reshetikhin-Turaev \cite{RT1, RT2}.

The construction of the unitary equivalence in \cite{Ip14} utilizes the language of \emph{multiplier Hopf algebra} and the GNS representations of $C^*$-algebra, which leads to the existence of a unitary operator $W$ called the \emph{multiplicative unitary} giving the desired intertwiner \eqref{decomp}. However, the construction of $W$ on the $C^*$-algebraic level is quite abstract and it is very difficult to write down explicitly the intertwiner as a product of quantum dilogarithm functions. We presented several examples in \cite{Ip14} for type $A_n$ and hint at a relationship to the Heisenberg double, but we were not able to generalize it to other types nor write down the formula in general. 

\subsection{Cluster realization of positive representations}
In this paper, we reprove Theorem \ref{mainthm} using a new technique which comes with the discovery of the cluster realization of quantum groups in \cite{SS16} for type $A_n$ and \cite{Ip16} for general types, where we found an embedding of the Drinfeld's double of the Borel part $\cD(\cU_q(\fb))\inj \cX_{D_{2,1}}$ into a quantum cluster algebra associated to a quiver $Q_{D_{2,1}}$ on the triangulation of a once punctured disk with two marked points. The quiver itself has a geometric meaning representing the Poisson structure of the moduli space of framed local systems for general groups \cite{FG2, Le}. A polarization of the quantum cluster variables, i.e. a choice of representations by the canonical variables on some Hilbert space $L^2(\R^N)$ through the exponential functions, recovers the positive representations $\cP_\l$ of $\cU_{q\til{q}}(\g_\R)$, which is the quotient of $\cD(\cU_q(\fb))$ by identifying the Cartan part $K_iK_i'=1$.

In particular, the generators of $\cU_q(\g)$ can be represented combinatorially using certain paths on $Q_{D_{2,1}}$ representing a telescoping sums of quantum cluster variables, while the coproduct is given by the quiver $Q_{D_{2,2}}$ on twice-punctured disk with two marked points, which is simply concatenating two copies of $Q_{D_{2,1}}$ along one edge. Furthermore, cluster mutations correspond to unitary equivalence of the representations $\cP_\l$. Finally, to each triangulation, we constructed explicitly in \cite{Ip16} using the universal $R$ matrix the sequence of cluster mutations $\mu$ realizing quiver mutations associated to the flip of triangulation as in Figure \ref{QQ}.

\begin{figure}[!htb]
\centering
\begin{tikzpicture}[scale=0.5, every node/.style={inner sep=0, minimum size=0.5cm, thick},x=0.3cm,y=0.3cm]
\begin{scope}[shift={(-14,0)}]
\draw (0,10)-- (-10,0)-- (0,-10)-- (10,0)--(0,10)--(0,-10);
\node (Q) at (-5,0) [minimum size=2cm]{ $Q$};
\node (Q) at (5,0) [minimum size=2cm]{ $Q'$};
\end{scope}
\begin{scope}[shift={(14,0)}]
\draw (10,0)-- (0,-10) -- (-10,0)--(0,10)--(10,0)--(-10, 0);
\node (Q) at (0,-5) [minimum size=2cm]{$Q$};
\node (Q) at (0,5) [minimum size=2cm]{$Q'$};
\end{scope}
\path (-2,0) edge[->, thick] (2,0);
\node at (0,1.5) {$\mu$};
\end{tikzpicture}
\caption{Flip of triangulation, with $Q,Q'$ the associated quivers.}\label{QQ}
\end{figure}
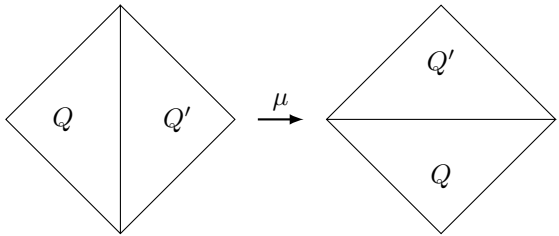

In this new language, it turns out that the restriction $\cP_\l^\fb$ can be easily describe by the so-called \emph{$F_i$-paths} as shown schematically in red below in Figure \ref{PP}. A flip of triangulation in this case creates a self-folded triangle, but since for the study of $\cP_\l^\fb$, the mutation sequence does not involve the gluing edge $\cE$ in the quiver, we can relax by un-gluing the edges $\cE$ and consider only a normal flip of triangulation as in Figure \ref{flip} of Section \ref{sec:flip}, where we prove the main result:

\begin{Thm} The sequence of cluster mutations realizing the flip of triangulation preserves the $F_i$-paths representing the restriction of positive representations to $\cU_{q\til{q}}(\fb_\R)$.
\end{Thm}

\begin{figure}[!htb]
\centering 
\begin{tikzpicture}[baseline=(a),x=0.7cm,y=0.7cm]
\draw [vthick] (0,0) circle (2);
\node (a) at (0,0) [draw, circle, minimum size=0.2, inner sep=1.5]{};
\node (b) at (0,2) [draw, circle, minimum size=0.2, inner sep=1.5, fill=black]{};
\node (c) at (0,-2) [draw, circle, minimum size=0.2, inner sep=1.5, fill=black]{};
\path (a) edge[-, thin] (b);
\path (a) edge[-, thin] (c);
\node at (0,-3){$\cP_\l^\fb$};
\node at (-2.7,-0.5)[red]{$\cF_{in}$};
\node at (2.7,-0.5)[red]{$\cF_{out}$};
\node at (-2,2){$\over[Q]$};
\node at (2,2){$Q$};
\node at (-0.2,1){$\cE$};
\draw[->,dashed,red, thick] (-2,0)--(0,-1);
\draw[->,dashed,red, thick] (0,-1)--(2,0);
\end{tikzpicture}
$\simeq_{\Phi_1}$
\begin{tikzpicture}[baseline=(a),x=0.7cm,y=0.7cm]
\draw [vthick] (0,0) circle (2);
\node (a) at (0,0) [draw, circle, minimum size=0.2, inner sep=1.5]{};
\node (b) at (0,2) [draw, circle, minimum size=0.2, inner sep=1.5, fill=black]{};
\node (c) at (0,-2) [draw, circle, minimum size=0.2, inner sep=1.5, fill=black]{};
\node at (0.3,0.7){$\l$};
\node at (0,-3){$\til{\cP^{\fb}}$};
\path (a) edge[-, thin] (b);
\node at (-2.7,-0.5)[red]{$\cF_{in}$};
\node at (2.7,-0.5)[red]{$\cF_{out}$};
\draw [thin](b) to [out=225, in=180](0,-0.7) to [out=0,in=-45](b);
\path (-2,0)edge[->, dashed,red, thick, bend right=60](2,0);
\end{tikzpicture}
\caption{Flipping to self-folded triangle, where the $F_i$-paths are preserved.}\label{PP}
\end{figure}
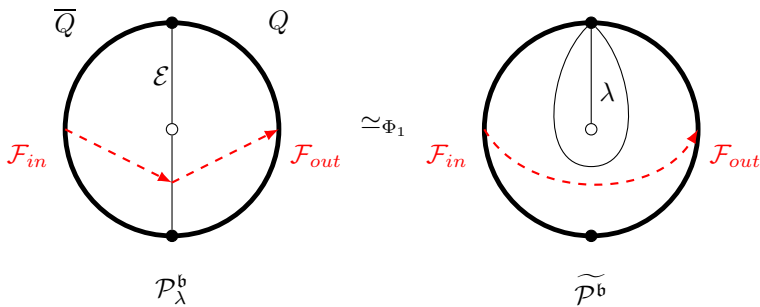

In particular, this shows that $\cP_\l^\fb$ is unitary equivalent to a representation $\til{\cP^\fb}$ representing the generators on the \emph{basic quiver} $Q_{D_{3,0}}$ associated only to a single triangle. In \cite{Ip14} we call this the \emph{standard form}, which is obtained by omitting half of the operators in the explicit representation of $\cP_\l^\fb$, and coincides with the so-called \emph{Feigin's homomorphism}.

Finally, the tensor product realized by concatenating two copies of $Q_{D_{2,1}}$, can be decomposed using three flips of triangles as shown schematically in Figure \ref{PPPM} (which is the same as the first three steps in the proof of tensor product decomposition of $\cP_\l$ for the whole quantum group in type $A_n$ constructed in \cite{SS17}). Since we found in \cite{Ip16} an explicit formula for the cluster mutations using a product of quantum dilogarithms, we \emph{completely solve the combinatorial problem} posed in \cite{Ip14} for explicitly writing down the tensor product decomposition of positive representations restricted to $\cU_{q\til{q}}(\fb_\R)$ of \emph{all types}, which cannot be done previously without the cluster realization. 

\begin{figure}[!htb]
\centering
\begin{tikzpicture}[baseline=(a),x=0.4cm,y=0.4cm]
\draw [vthick] (0,0) circle (2);
\node (a) at (-1,0) [draw, circle, minimum size=0.2, inner sep=1.5]{};
\node (d) at (1,0) [draw, circle, minimum size=0.2, inner sep=1.5]{};
\node (b) at (0,2) [draw, circle, minimum size=0.2, inner sep=1.5, fill=black]{};
\node (c) at (0,-2) [draw, circle, minimum size=0.2, inner sep=1.5, fill=black]{};
\path (a) edge[-, thin] (b);
\path (a) edge[-, thin] (c);
\path (b) edge[-, thin] (d);
\path (c) edge[-, thin] (d);
\path (b) edge[-, thin] (c);
\node at (0,-3){$\cP_\l^\fb\ox \cP_\mu^\fb$};
\node at (-3,-0.5)[red]{$\cF_{in}$};
\node at (3,-0.5)[red]{$\cF_{out}$};
\draw[->,dashed,red,thick] (-2,0)--(0,-1.3);
\draw[->,dashed,red, thick] (0,-1.3)--(2,0);
\end{tikzpicture}
$\simeq_{\Phi_1\ox \Phi_1}$
\begin{tikzpicture}[baseline=(a),x=0.4cm,y=0.4cm]
\draw [vthick] (0,0) circle (2);
\node (a) at (-1,0) [draw, circle, minimum size=0.2, inner sep=1.5]{};
\node (d) at (1,0) [draw, circle, minimum size=0.2, inner sep=1.5]{};
\node (b) at (0,2) [draw, circle, minimum size=0.2, inner sep=1.5, fill=black]{};
\node (c) at (0,-2) [draw, circle, minimum size=0.2, inner sep=1.5, fill=black]{};
\path (a) edge[-, thin] (b);
\path (b) edge[-, thin] (d);
\draw [thin](b) to [out=225, in=138](-1.25,-0.5) to [out=-45,in=-105](b);
\draw [thin](b) to [out=-75, in=225](1.25,-0.5) to [out=45,in=-38](b);
\path (b) edge[-, thin] (c);
\node at (0,-3){$\til{\cP^{\fb}}\ox \til{\cP^{\fb}}$};
\node at (-3,-0.5)[red]{$\cF_{in}$};
\node at (3,-0.5)[red]{$\cF_{out}$};
\path (-2,0)edge[->, dashed,red, thick, bend right=30](0,-1.5);
\path (0,-1.5)edge[->, dashed,red, thick, bend right=30](2,0);
\end{tikzpicture}
$\simeq_{\Phi_3}$
\begin{tikzpicture}[baseline=(a),x=0.4cm,y=0.4cm]
\draw [vthick] (0,0) circle (2);
\node (a) at (-1,0) [draw, circle, minimum size=0.2, inner sep=1.5]{};
\node (d) at (1,0) [draw, circle, minimum size=0.2, inner sep=1.5]{};
\node (b) at (0,2) [draw, circle, minimum size=0.2, inner sep=1.5, fill=black]{};
\node (c) at (0,-2) [draw, circle, minimum size=0.2, inner sep=1.5, fill=black]{};
\path (a) edge[-, thin] (b);
\path (b) edge[-, thin] (d);
\draw [thin](b) to [out=225, in=138](-1.25,-0.5) to [out=-45,in=-105](b);
\draw [thin](b) to [out=-75, in=225](1.25,-0.5) to [out=45,in=-38](b);
\draw [thin](b) to [out=205, in=90](-1.75,0) to [out=-90,in=180] (0,-1) to [out=0, in=-90](1.75,0)to [out=90, in=-25](b);
\node at (0,-3){$\til{\cP^{\fb}}\ox \cM$};
\node at (-3,-0.5)[red]{$\cF_{in}$};
\node at (3,-0.5)[red]{$\cF_{out}$};
\path (-2,-0.3)edge[-, dashed,red, thick, bend right=30](0,-1.5);
\path (0,-1.5)edge[->, dashed,red, thick, bend right=30](2,-0.3);
\end{tikzpicture}
\caption{Tensor product decomposition $\cP_\l^\fb\ox \cP_\l^\fb\simeq \til{\cP^\fb}\ox \cM$.}\label{PPPM}
\end{figure}

\subsection{Outline of the paper}
The paper is organized as follows. In Section \ref{sec:prelim}, we recall the positive representations $\cP_\l$ and some identities of the quantum dilogarithm function $g_b$. In Section \ref{sec:cluster}, we recall the definitions and results concerning quantum torus algebra $\cX^\bi$. In Section \ref{sec:basic}, we reconstruct the basic quiver $Q^{\bi}$ in \cite{Ip16} in the notation of the current paper, and in Section \ref{sec:realization} we recall the cluster realization of positive representations. Finally in Section \ref{sec:flip}, we prove the main result and in Section \ref{sec:ex} we give several examples to demonstrate the algorithms of the construction of the tensor product decomposition of positive representations restricted to the Borel part $\cU_{q\til{q}}(\fb_\R)$.

\section*{Acknowledgments}
I would like to thank Gus Schrader, Alexander Shapiro and Ian Le for valuable discussions and hospitality at University of Toronto and Perimeter Institute at which this work is inspired. This work is partly based on the talks given at CQIS-2017 held in Dubna, Russia, and the 2017 Autumn Meeting of the Mathematics Society of Japan. This work is supported by Top Global University Project, MEXT, Japan at Kyoto University, and JSPS KAKENHI Grant Numbers JP16K17571.


\section{Preliminaries}\label{sec:prelim} In this section, let us recall several notations and definitions that will be used throughout the paper. We will follow mostly the convention used in \cite{Ip14} and \cite{Ip16}. 
\subsection{Definition of the modular double $\cU_{q\til{q}}(\g_\R)$}\label{sec:prelim:Uqgr}
Let $\g$ be a simple Lie algebra over $\C$, $\cI=\{1,2,...,n\}$ denotes the set of nodes of the Dynkin diagram of $\g$ where $n=rank(\g)$. Let $\{\a_i\}_{i\in \cI}$ be the set of positive simple roots, and let $w_0\in W$ be the longest element of the Weyl group of $\g$, where $N:=l(w_0)$ is the length of the longest word. We call a sequence $$\bi=(i_1,...,i_N)\in\cI^N$$ a reduced word of $w_0$ if $w_0=s_{i_1}...s_{i_N}$ is a reduced expression, where $s_{i_k}$ are the simple reflections of the root space. We will denote the reversed word by $$\over[\bi]:=(i_N,...,i_1).$$

\begin{Def} \label{qi} Let $q$ be a formal parameter. Let $(-,-)$ be the $W$-invariant inner product of the root lattice, and we define 
$$a_{ij}:=\frac{2(\a_i,\a_j)}{(\a_i,\a_i)},$$
such that $A:=(a_{ij})$ is the \emph{Cartan matrix}. 

We normalize $(-,-)$ as follows: for $i\in\cI$, we define the \emph{multipliers} to be
\Eq{\bd_i:=\frac{1}{2}(\a_i,\a_i):=
\case{1&\mbox{$i$ is long root or in the simply-laced case,}\\
\frac{1}{2}&\mbox{$i$ is short root in type $B,C,F$,}\\
\frac{1}{3}&\mbox{$i$ is short root in type $G_2$,}}
}
where $(\a_i,\a_j)=-1$ when $i,j$ are adjacent in the Dynkin diagram, such that
$$\bd_ia_{ij}=\bd_ja_{ji}.$$ We then define $q_i:=q^{\bd_i}$, which we will also write as
\Eq{
q_l&:=q,\\ 
q_s&:=\case{q^{\frac12}&\mbox{$\g$ is of type $B_n, C_n, F_4$},\\q^{\frac13}&\mbox{$\g$ is of type $G_2$},}
}
for the $q$ parameters corresponding to long and short roots respectively.
\end{Def}
\begin{Def}\cite{D,J} 
The Drinfeld-Jimbo quantum group $\cU_q(\g_\R)$ is the Hopf algebra generated by $\{E_i,F_i,K_i^{\pm1}\}_{i\in \cI}$ over $\C$ subjected to the relations for $i,j\in \cI$:
\Eq{
K_iE_j=q_i^{a_{ij}}E_jK_i,\tab K_iF_j=q_i^{-a_{ij}}F_jK_i,\tab {[E_i,F_j]} = \d_{ij}\frac{K_i-K_i\inv}{q_i-q_i\inv},
}
together with the Serre relations for $i\neq j$:
\Eq{
\sum_{k=0}^{1-a_{ij}}(-1)^k\frac{[1-a_{ij}]_{q_i}!}{[1-a_{ij}-k]_{q_i}![k]_{q_i}!}X_i^{k}X_jX_i^{1-a_{ij}-k}&=0,\tab X=E,F,\label{SerreE}
}
where $[k]_q:=\frac{q^k-q^{-k}}{q-q\inv}$. 

The Hopf algebra structure of $\cU_q(\g)$ is given by 
\Eq{
\D(E_i)=&1\ox E_i+E_i\ox K_i,\\
\D(F_i)=&F_i\ox 1+K_i^{-1}\ox F_i,\\
\D(K_i)=&K_i\ox K_i.
}
We will not need the counit and antipode in this paper.
\end{Def}

In the split real case, it is required that $|q|=1$. Throughout the paper, we let 
\Eq{q:=e^{\pi \sqrt{-1} b^2}} with $0<b^2<1$ and $b^2\in\R\setminus\Q$. We also write $$q_i:=e^{\pi \sqrt{-1} b_i^2}$$ such that 
\Eq{
b_i:=\case{b_l:=b&\mbox{$\a_i$ is long root or $\g$ is simply-laced,}\\
b_s:=\sqrt{\bd_i}b&\mbox{$\a_i$ is short root.}
}
} 

We define $\cU_q(\g_\R)$ to be the real form of $\cU_q(\g)$ induced by the star structure
\Eq{E_i^*=E_i,\tab F_i^*=F_i,\tab K_i^*=K_i.}
Finally, from the results of \cite{Ip2,Ip3}, let $\til{q}:=e^{\pi \sqrt{-1} b_s^{-2}}$ and we define the modular double to be
\Eq{\cU_{q\til{q}}(\g_\R):=\cU_q(\g_\R)\ox \cU_{\til{q}}(\g_\R)&\tab \mbox{$\g$ is simply-laced,}\\
\cU_{q\til{q}}(\g_\R):=\cU_q(\g_\R)\ox \cU_{\til{q}}({}^L\g_\R)&\tab \mbox{otherwise,}}
where ${}^L\g_\R$ is the Langlands dual obtained by interchanging the long and short roots of $\g_\R$.

\subsection{Positive representations of $\cU_{q\til{q}}(\g_\R)$}\label{sec:prelim:pos}
In \cite{FI, Ip2,Ip3}, a special class of representations for $\cU_{q\til{q}}(\g_\R)$, called the positive representations, is defined. The generators of $\cU_{q\til{q}}(\g_\R)$ are realized by positive essentially self-adjoint operators on certain Hilbert space, and satisfy the \emph{transcendental relations} \eqref{transdef}. In particular the quantum group and its modular double counterpart are represented on the same Hilbert space, generalizing the situation of $\cU_{q\til{q}}(\sl(2,\R))$ introduced in \cite{Fa1, Fa2} and studied in \cite{PT2}. More precisely, 
\begin{Thm}\cite{FI, Ip2, Ip3} Define the rescaled generators to be
\Eq{\be_i:=2\sin(\pi b_i^2)E_i,\tab \bf_i:=2\sin(\pi b_i^2)F_i.\label{smallef}}
There exists a family of representations $\cP_{\l}$ of $\cU_{q\til{q}}(\g_\R)$ parametrized by the $\R_+$-span of the cone of positive weights $\l\in P_\R^+$, or equivalently by $\l\in \R_+^{rank(\g)}$, such that 
\begin{itemize}
\item The generators $\be_i,\bf_i,K_i$ are represented by positive essentially self-adjoint operators acting on $L^2(\R^{N})$ where $N=l(w_0)$.
\item Define the transcendental generators:
\Eq{\til{\be_i}:=\be_i^{\frac{1}{b_i^2}},\tab \til{\bf_i}:=\bf_i^{\frac{1}{b_i^2}},\tab \til{K_i}:=K_i^{\frac{1}{b_i^2}}.\label{transdef}}
\begin{itemize}
\item if $\g$ is simply-laced, the generators $\til{\be_i},\til{\bf_i},\til{K_i}$ are obtained by replacing $b$ with $b\inv$ in the representations of the generators $\be_i,\bf_i,K_i$. 
\item If $\g$ is non-simply-laced, then the generators $\til{E_i},\til{F_i},\til{K_i}$ with $\til{\be_i}:=2\sin(\pi b_i^{-2})\til{E_i}$ and $\til{\bf_i}:=2\sin(\pi b_i^{-2})\til{F_i}$ generate $\cU_{\til{q}}({}^L\g_\R)$.
\end{itemize}
\item The generators $\be_i,\bf_i,K_i$ and $\til{\be_i},\til{\bf_i},\til{K_i}$ commute weakly up to a sign.
\end{itemize}
\end{Thm}

Let $\cP_\l\simeq L^2(\R^N, du_1...du_N)$, and let $u_k, p_k:=\frac{1}{2\pi\sqrt{-1}}\del[,u_k]$ be the standard position and momentum operator respectively acting on $L^2(\R^N)$ as self-adjoint operators. Then the representations can be constructed explicitly as follows:
\begin{Thm}\label{FKaction}\cite{Ip2,Ip3} For a fixed reduced word $\bi=(i_1,...,i_N)$, the positive representation $\cP_\l^{\bi}$ parametrized by $\l=(\l_1,...,\l_n)\in\R_{\geq 0}^n$ is given by positive essentially self-adjoint operators
\Eq{
K_i&=\exp\left(-2\pi b_i\l_i-\pi \sum_{k=1}^N a_{i,i_k} b_{i_k}u_k\right),\label{KK}\\
\bf_i&=\bf_i^-+\bf_i^+\label{FF}\\
&:=\sum_{k:i_k=i}\bf^{k,-}+\sum_{k:i_k=i}\bf^{k,+},\nonumber
}
acting on $L^2(\R^N)$, where 
\Eq{\bf^{k,\pm} := \exp\left(\pm\left(\sum_{j=1}^{k-1} \pi b_{i_j}a_{i_j,i_k} u_j + \pi b_{i_k}u_k+2\pi b_{i_k}\l_{i_k}\right)+2\pi b_{i_k} p_k\right).}
The $E_i$ generators corresponding to the right most root $i_N$ of $\bi$ is given explicitly by
\Eq{\label{EE} 
\be_{i_N}&=\be_{i_N}^-+\be_{i_N}^+\\
&:=e^{\pi b_{i_N} (u_N-2p_N)}+e^{\pi b_{i_N}(-u_N-2p_N)},\nonumber
}
while for the other generators we have in general
\Eq{
\be_i=\be_i^-+\be_i^+,
}
where $\be_i^\pm$ can be obtained from conjugation by quantum dilogarithms of $\be_{i_N}^\pm$ above, and can be realized explicitly as the $E_i$-paths polynomials on the cluster realization \cite{Ip16} (cf. Section \ref{sec:realization}).
\end{Thm}
\begin{Thm}
For any reduced words $\bi$ and $\bi'$, we have unitary equivalence 
\Eq{
\cP_\l\simeq \cP_\l^\bi\simeq \cP_\l^{\bi'}.
}
\end{Thm}
In particular, for any root index $i\in\cI$, we can choose $\bi$ with $i_N=i$ in order to observe that
\begin{Cor}\label{efcom} For any $i=1,..., n$, we have
\Eq{
\bf_i^-\bf_i^+&=q_i^{-2}\bf_i^+\bf_i^-,\tab \be_i^-\be_i^+=q_i^{-2}\be_i^+\be_i^-,\\
\frac{[\be_i^\pm, \bf_j^\mp]}{q_i-q_i\inv} &= \mp\d_{ij} K_i^{\pm1},\\
\bf^{k,\pm}\bf^{l,\pm} &= q_{i_k}^{\pm a_{kl}} \bf^{l,\pm}\bf^{k,\pm},\tab 1\leq l<k\leq N.
}
\end{Cor}
\begin{Rem}
We see that, for example, the triple $\{\be_i^-, \bf_i^+, K_i\inv\}$ forms the commutation relation of the \emph{Heisenberg double} \cite{Ka1}, which will be needed to prove the tensor product decomposition of the positive representations restricted to the Borel part in Section \ref{sec:flip}, thus confirming the intuition discussed in \cite{Ip14}.
\end{Rem}
\subsection{Quantum dilogarithm identities}\label{sec:prelim:qdlog}
We recall the quantum dilogarithm identities needed in this paper. More details can be found in \cite{FKa, Ip16}. The non-compact quantum dilogarithm is a meromorphic function that can be represented as an integral expression:
\Eq{
g_b(x):=\exp\left(\frac{1}{4}\int_{\R+i0} \frac{x^{\frac{t}{ib}}}{\sinh(\pi bt)\sinh(\pi b\inv t)}\frac{dt}{t}\right),
}
such that by functional calculus, it is unitary when $x$ is positive self-adjoint.
\begin{Rem}
During formal algebraic manipulation, one may consider its compact version instead, which in terms of formal power series are related by
\Eq{
g_b(x)\sim \Psi^q(x)\inv,
}
where the Faddeev-Kashaev's quantum dilogarithm is
\Eq{
\Psi^q(x):=\prod_{r=0}^\oo (1+q^{2r+1}x)\inv = Exp_{q^{-2}}\left(\frac{u}{q-q\inv}\right),
}
with
\Eq{
Exp_q(x):=\sum_{k\geq 0} \frac{x^k}{(k)_q!},\tab (k)_q:=\frac{1-q^k}{1-q}.
}
\end{Rem}

We will only need the following identities in this paper.
\begin{Lem}[Quantum dilogarithm identities]
Let $u,v$ be positive self-adjoint variables. If $uv=q^2 vu$, then we have the quantum exponential relation:
\Eq{
g_b(u+v)&=g_b(u)g_b(v)\label{guv},
}
and the conjugation
\Eq{
g_b(v)ug_b(v)^*&=qvu+u,\label{gcon}\\
g_b(u)^*vg_b(u)&=v+qvu.\nonumber
}
We also have in the doubly-laced case ($\sqrt{2}b_s=b$)
\Eq{\label{gdouble}
g_{b_s}(u+v)=g_{b_s}(u)g_b(q\inv uv)g_{b_s}(v).
}
Let again $u,v$ be positive self-adjoint and define $$c:=\frac{[u,v]}{q-q\inv},$$ such that $uc=q^2cu$ and $cv=q^2vc$. Then we have the generalized equation:
\Eq{
g_b(v)u g_b^*(v)&=c+u,\label{g12}\\
g_b(u)^*v g_b(u)&=v+c,\nonumber
}
in which \eqref{gcon} is a special case. 
\end{Lem}
\section{Quantum cluster algebra}\label{sec:cluster}
In this section, let us recall the notation of the quantum torus algebra used in this paper. We will follow mostly the notations used in \cite{Ip16} and \cite{SS16} but with some modifications.
\subsection{Quantum torus algebra and quiver}\label{sec:cluster:torus}
\begin{Def}[Cluster seed]
A cluster seed is a datum $\bi=(I, I_0, B, D)$ where $I$ is a finite set, $I_0\subset I$ is a subset called the \emph{frozen subset}, $B=(b_{ij})_{i,j\in I}$ a skew-symmetrizable $\frac12\Z$-valued matrix called the \emph{exchange matrix}, and $D=diag(d_i)_{i\in I}$ is a diagonal $\Q$-matrix called the \emph{multiplier} such that $DB=-B^TD$ is skew-symmetric. 
\end{Def}
\begin{Nota} In the rest of the paper, we will write $\bi'=(I', I_0', B', D')$ and $\over[\bi] =(\over[I], \over[I_0], \over[B], \over[D])$ and the respective elements as $i'\in I', \over[i]\in\over[I]$ etc. for convenience.
\end{Nota}
\begin{Def}[Quantum torus algebra]
Let $q$ be a formal parameter. We define the \emph{quantum torus algebra} $\cX^{\bi}$ associated to a cluster seed $\bi$ to be an associative algebra over $\C[q^d]$, where $d=\min_{i\in I}(d_i)$, generated by $\{X_i\}_{i\in I}$ subject to the relations
\Eq{\label{XiXj}
X_iX_j=q^{-2w_{ij}}X_jX_i,\tab i,j\in I,
}
where 
\Eq{w_{ij}=d_ib_{ij}=-w_{ij}.} The generators $X_i\in \cX^{\bi}$ are called the \emph{quantum cluster variables}, and they are \emph{frozen} if $i\in I_0$. We denote by $\bT^\bi$ the non-commutative field of fraction of $\cX^\bi$.
\end{Def}

Alternatively, given $B$ and $D$ as above, let $\L_\bi$ be a lattice with basis $\{e_i\}_{i\in I}$, and define a skew symmetric $d\Z$-valued form on $\L_\bi$ by $(e_i, e_j):=w_{ij}$.
Then $\cX^\bi$ is generated by $\{X_{\l}\}_{\l\in \L_\bi}$ with $X_0:=1$ subject to the relations
\Eq{q^{(\l,\mu)}X_{\l}X_{\mu} = X_{\l+\mu}.}
\begin{Nota}
Under this realization, we shall write $X_i=X_{e_i}$, and define the notation
\Eq{X_{i_1,...,i_k}:=X_{e_{i_1}+...+e_{i_k}},}
or more generally for $n_1,...,n_k\in \Z$,
\Eq{X_{i_1^{n_1},...,i_k^{n_k}}:=X_{n_1e_{i_1}+...+n_ke_{i_k}}.}
\end{Nota}
Let $c_{ij}\in\half\Z$ for $i,j\in I$ be defined by $c_{ij}=\case{b_{ij}&\mbox{if $d_i=d_j,$}\\w_{ij}&\mbox{otherwise.}}$
\begin{Def}[Quiver associated to $\bi$]\label{quiver} We associate to each seed $\bi=(I,I_0,B,D)$ a quiver $Q^\bi$ with vertices labeled by $I$ and adjacency matrix $(c_{ij})_{i,j\in I}$. We call $i\in I$ a \emph{short (resp. long) node} if $q^{d_i}=q_s$ (resp. $q^{d_i}=q$). An arrow $i\to j$ represents the algebraic relation 
\Eq{X_iX_j=q_*^{-2}X_jX_i} where $q_*=q_s$ if both $i,j$ are short nodes, or $q_*=q$ otherwise.
\end{Def}
We will use squares to denote frozen nodes $i\in I_0$ and circles otherwise. We will also use dashed arrow if $c_{ij}=\half$, which only occurs between frozen nodes. 
We will represent the algebraic relations \eqref{XiXj} by thick or thin arrows (see Figure \ref{thick}) for display convenience (thickness is \emph{not} part of the data of the quiver). Thin arrows only occur in the non-simply-laced case between two short nodes.
   
\begin{figure}[H]
\centering
  \begin{tikzpicture}[every node/.style={inner sep=0, minimum size=0.5cm, thick}, x=1cm, y=1cm]
\node[draw,circle] (i1) at(0,4) {$i$};
  \node[draw,circle] (j1) at (3,4) {$j$};
   \draw[vthick, ->](i1) to (j1);
\node at (6,4) {$X_iX_j = q^{-2} X_j X_i$}; 
\node[draw] (i2) at(0,3) {$i$};
  \node[draw] (j2) at (3,3) {$j$};
   \draw[vthick, dashed, ->](i2) to (j2);
\node at (6,3) {$X_iX_j = q^{-1} X_j X_i$};
\node[draw,circle] (i3) at(0,2) {$i$};
  \node[draw,circle] (j3) at (3,2) {$j$};
   \draw[thin, ->](i3) to (j3);
\node at (6,2) {$X_iX_j = q_s^{-2} X_j X_i$};
  \end{tikzpicture}
  \caption{Arrows between nodes and their algebraic meaning.}\label{thick}
  \end{figure}
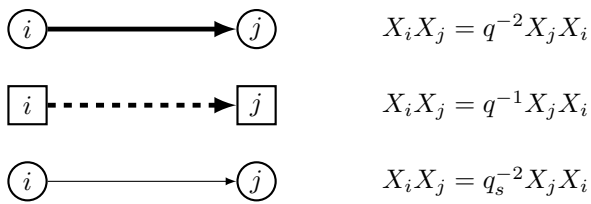
  
\begin{Def}\label{posrepX} A \emph{positive representation} of the quantum torus algebra $\cX^\bi$ on a Hilbert space $\cH=L^2(\R^M)$ is an assignment
\Eq{
X_i=e^{2\pi b L_i},\tab i\in I,
}
where $L_i:=L_i(u_k, p_k, \l_k)$ is a linear combination of the position and momentum operators  $\{u_k, p_k\}_{k=1}^M$ and complex parameters $\l_k$ with 
\Eq{[L_i, L_j]=\frac{w_{ij}}{2\pi\sqrt{-1}},} such that $X_i$ acts as a positive self-adjoint operator on $\cH$.
\end{Def}
\subsection{Quantum cluster mutation}\label{sec:cluster:mutate}
Next we define the cluster mutations of a seed and its quiver, and the quantum cluster mutations for the algebra. Here we will use the notion that keeps the indexing $I$ of the seeds, which ensures the consistency of the relation $\mu_k^2=\mathrm{Id}$.
\begin{Def}[Cluster mutation] Given a pair of seeds $\bi, \bi'$ with $I=I', I_0=I_0'$, and an element $k\in I\setminus I_0$, a \emph{cluster mutation in direction $k$} is an isomorphism $\mu_k:\bi\to \bi'$ such that $\mu_k(i)=i$ for all $i\in I$, $d_i'=d_i$, and
\Eq{
b'_{ij} &= \case{-b_{ij}&\mbox{if $i=k$ or $j=k$},\\ b_{ij}+\frac{b_{ik}|b_{kj}|+|b_{ik}|b_{kj}}{2}&\mbox{otherwise}.}
}

Then the quiver mutation $Q^\bi\to Q^{\bi'}$ corresponding to the mutation $\mu_k$ can be performed by the well-known rule (with obvious generalization for dashed arrows)
\begin{itemize}
\item[(1)] reverse all the arrows incident to the vertex $k$;
\item[(2)] for each pair of arrows $i\to k$ and $k\to j$, add $n_{ij}^k$ arrows from $i\to j$, where $n_{ij}^k=d_k\inv$ if $k$ is a short node and both $i,j$ are long nodes, or $n_{ij}^k=1$ otherwise;
\item[(3)] delete any 2-cycles.
\end{itemize}
\end{Def}
\begin{Def}[Quantum cluster mutation]
The cluster mutation $\mu_k:\bi\to \bi'$, induces an isomorphism $\mu_k^q:\bT^{\bi'}\to \bT^{\bi}$ called the \emph{quantum cluster mutation}, which can be written as a composition of two homomorphisms
\Eq{\label{mudecomp}
\mu_k^q=\mu_k^\#\circ \mu_k',
}
where $\mu_k':\bT^{\bi'}\to \bT^\bi$ is a monomial transformation defined by
\Eq{
\mu_k'(\what{X}_i):=\case{X_k\inv&\mbox{if $i=k$},\\ X_i&\mbox{if $i\neq k$ and $b_{ki}\leq 0$},\\ q_i^{b_{ik}b_{ki}}X_iX_k^{b_{ki}}&\mbox{if $i\neq k$ and $b_{ki}\geq 0$},}
}
and $\mu_k^\#:\bT^\bi\to\bT^\bi$ is a conjugation by the \emph{quantum dilogarithm function}
\Eq{
\mu_k^\#:=Ad_{g_{b_k}^*(X_k)}, 
}
where $b_k=\sqrt{d_k}b$.
\end{Def}

It is also useful to recall the following lemma from \cite[Lemma 1.1]{SS16}:
\begin{Lem}\label{useful} Let $\mu_{i_1}, ... ,\mu_{i_k}$ be a sequence of mutation, and denote the intermediate seeds by $\bi_j:=\mu_{i_j}\cdots\mu_{i_1}(\bi)$. Then the induced quantum cluster mutation $\mu_{i_1}^q \cdots\mu_{i_k}^q: \bT^{\bi_k}\to \bT^{\bi}$ can be written as 
\Eq{
\mu_{i_1}^q \cdots\mu_{i_k}^q=\Phi_k\circ M_k,
}
where $M_k:\bT^{\bi_k}\to \bT^{\bi}$ and $\Phi_k:\bT^\bi\to\bT^\bi$ are given by
\Eq{
M_k&:= \mu_{i_1}'\mu_{i_2}'\cdots \mu_{i_k}',\\
\Phi_k&:= Ad_{g_{b_{i_1}}^*(X_{i_1})}Ad_{g_{b_{i_2}}^*(M_1(X_{i_2}^{(1)}))}\cdots Ad_{g_{b_{i_k}}^*(M_{k-1}(X_{i_k}^{(k-1)})},
}
and $X^{(j)}_i\in \cX^{\bi_j}$ denotes the corresponding quantum cluster variables of the algebra $\cX^{\bi_j}$. 
\end{Lem}

\subsection{Amalgamation}\label{sec:cluster:amal}
Finally we briefly recall the procedure of \emph{amalgamation} of two quantum torus algebra \cite{FG2}:
\begin{Def}\label{amal}
Let $\bi, \bi'$ be two cluster seeds and let $\cX^{\bi}, \cX^{\bi'}$ be the corresponding quantum torus algebra. Let $J\subset I_0$ and $J'\subset I_0'$ be subsets of the frozen nodes with a bijection $\phi:J\to J'$ such that $$d_{\phi(i)}'=d_i,\tab i\in J.$$ Then the amalgamation of $\cX^{\bi}$ and $\cX^{\bi'}$ along $\phi$ is identified with the subalgebra $\til{\cX}\sub\cX^{\bi}\ox \cX^{\bi'}$ generated by the variables $\{\til{X}_i\}_{i\in I\cup I'}$ where 
\Eq{
\til{X}_i&:= \case{
X_i\ox 1&\mbox{if $i\in I\setminus J$},\\
1\ox X_i'&\mbox{if $i\in I'\setminus J'$},
}\\\nonumber
\til{X}_i= \til{X}_{\phi(i)} &:= X_i\ox X_{\phi(i)}'\tab i\in J.
}
\end{Def}

Equivalently, the amalgamation of the corresponding quivers $Q,Q'$ is a new quiver $\til{Q}$ constructed by gluing the vertices along $\phi$, defreezing those vertices that are glued, and removing any resulting 2-cycles.
\section{The basic quivers}\label{sec:basic}
In \cite{Ip16}, we constructed the basic quiver $Q$ associated to a triangle such that the quantum group can be embedded into an amalgamation of $Q$ and its mirror image $\over[Q]$ (See Figure \ref{PP}). Let us recall its construction relevant to this paper, with a slightly different indexing.

First of all, following \cite{BFZ}, we define the notation $k^+$ and $k^-$ as follows
\begin{Nota}
Given a reduced word $\bi=(i_1,...,i_N)$, for $k\in\{1,...,N\}$, we denote by $k^+$ the smallest index $l>k$ such that $i_k = i_l$ if it exists, or $N+i_k$ otherwise.

Similarly, for $k\in \{1,...,N+n\}$, we denote by $k^-$ the largest index $l<k$ such that $l^+=k$ if it exists, or $0$ otherwise. Finally we define $i_{N+j} := j\in \cI$ for $j=1,...,n$.

We define the set of extremal indices by 
\Eq{\cF_{in}:=\{k: k^-=0\},\tab \cF_{out}:=\{N+1,..., N+n\},}
and let $k^*:=(N+k)^-$ be the largest index which is not extremal and such that $i_{k^*}=k$.
\end{Nota}

By abuse of notation we will also use $\bi$ to denote the cluster seed corresponding to the quiver $Q^{\bi}$, called the \emph{basic quiver}, described as follows.
\begin{Def}\label{basic}The \emph{basic quiver} $Q^\bi$ has nodes $I=\{1,...,N+2n\}$. It contains a subquiver $Q_F^\bi$ with $N+n$ nodes $\{1,2,...,N+n\}$ where the frozen nodes are $\cF_{in}\cup \cF_{out}$, such that for $k=1,...,N$:
\begin{itemize}
\item there is a single arrow $k\to k^+$,
\item there is a single arrow $l\to k$ if $l^-<k^-<l<k$,
\end{itemize}
and for the frozen nodes,
\begin{itemize}
\item there is a half arrow $k\DashedArrow[densely dashed    ]  l$ if $k,l\in \cF_{in}$ and $k>l$,
\item there is a half arrow $k\DashedArrow[densely dashed    ]  l$ if $k,l\in \cF_{out}$ and $k^->l^-$.
\end{itemize} 
The multiplier $D=(d_i)$ of the cluster seed $\bi$ is defined to be $d_k := \bd_{i_k}$.
\end{Def}
We attach the frozen nodes to two sides of the triangle, and we will usually display the arrows $k\to k^+$ in $Q_F^\bi$ in horizontal rows. The full quiver $Q^\bi$ is constructed by adding $n$ more frozen nodes $$\cE:=\{N+n+1,...,N+2n\}$$ attached to the third side with multiplier $d_{N+n+k}:=\bd_k$, and with additional arrows between them and $Q_F^\bi$, but in this paper we do not need them, so we will not review the construction. Alternatively, the subquiver $Q_F^\bi$ can also be constructed from building blocks called the \emph{elementary quivers}, see e.g. \cite{Ip16, Le}.

\begin{Ex} Our running example will be the quiver for type $A_3$ (Figure \ref{A3q}) and type $B_3$ (Figure \ref{B3q}). For completeness we will show the full quiver $Q^{\bi}$, but grayed out the part of the quiver outside $Q_F^\bi$ which is not needed in the construction of this paper. We highlighted the $F_i$-paths (cf. Definition \ref{Fipath}) in red.
\end{Ex}

\begin{Def} Given two basic quivers $Q^\bi, Q^{\bi'}$ associated to two reduced words $\bi,\bi'$, we denote by $Q^{\bi\bi'}$ the quiver obtained by amalgamation along $k\mapsto \phi(k)$ where $k\in \cF_{out}\sub Q^\bi$ and $\phi(k)\in \cF_{in}\sub Q^{\bi'}$ such that $i'_{\phi(k)}=i_k$. We denote by $$\cX^{\bi\bi'}\sub \cX^\bi\ox \cX^{\bi'}$$ the corresponding quantum torus algebra. Similarly, we let $Q_F^{\bi\bi'}$ be the subquiver obtained by amalgamating $Q_F^{\bi}$ and $Q_F^{\bi'}$ along the same map $\phi$, and let $$\cX_F^{\bi\bi'}\sub\cX^{\bi\bi'}$$ be the corresponding subalgebra.
\end{Def}

\begin{Def}\label{path} Given a path $\cP=(i_1,...,i_m)$ on the quiver $Q^\bi$ with $i_k\to i_{k+1}$ for every $k$, we write 
\Eq{X_{\cP} := X_{i_1,...,i_m}} for the product of all cluster variables along the path, and we define the \emph{path polynomial} (note that we ignore the last index) by
\Eq{X(\cP):=X(i_1,...,i_{m-1}) := \sum_{k=1}^{m-1} X_{i_1,...,i_k}\in \cX^{\bi}.}
 Given two paths $\cP, \cP'$ with $i_m=i_1'$, we simply denoted by $$\cP\cP'=(i_1,...,i_m=i_1',...,i'_{m'})$$ the path in $Q^{\bi\bi'}$ obtained by concatenating the two paths.
\end{Def}
\begin{Def}\label{Fipath} An \emph{$F_i$-path} is the path on the quiver $Q^\bi$ which starts from $\cF_{in}$ and ends at $\cF_{out}$ along the horizontal path on the quiver with nodes $\{k: i_k = i\}$ (see e.g. Figure \ref{AB3q}).
\end{Def}

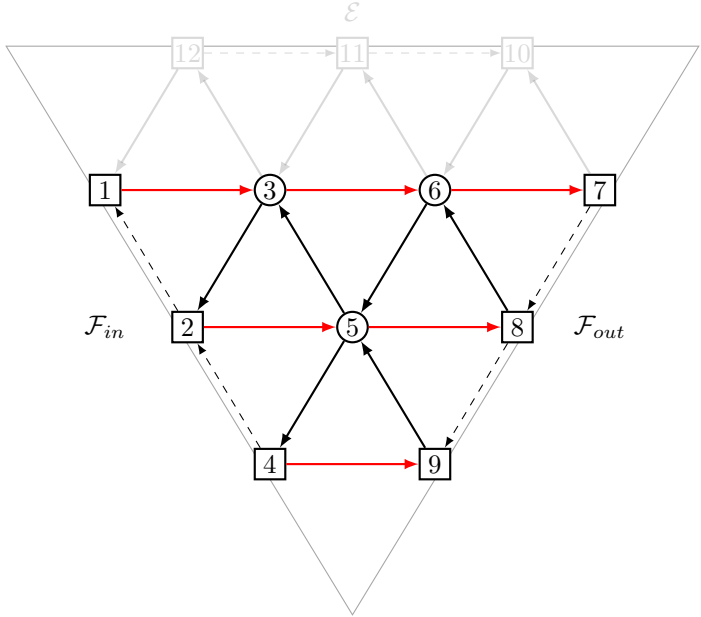
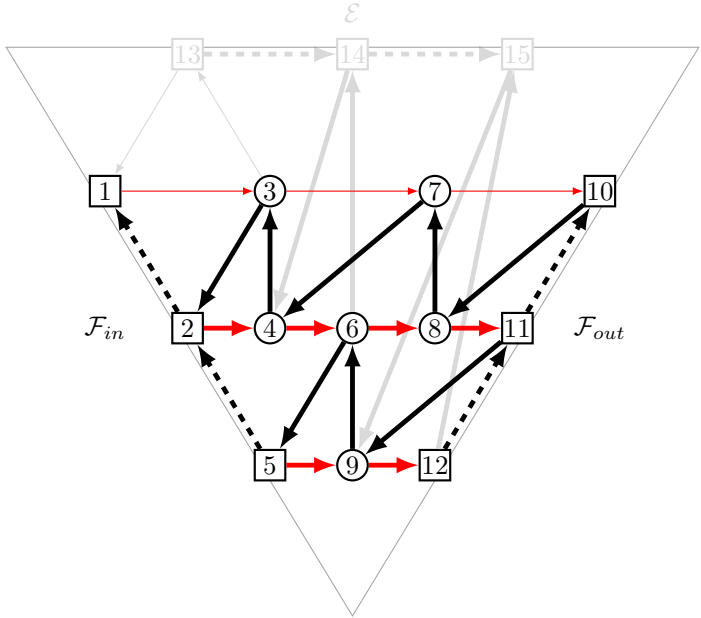
\begin{figure}[H]
\begin{subfigure}[H]{\textwidth}
\centering
\begin{tikzpicture}[scale=0.9, every node/.style={inner sep=0, minimum size=0.4cm, thick, fill=white, draw,circle}, x=1.2cm, y=1cm]
\draw [gray!70](-1.2,6.1)--(7.2,6.1)--(3,-2.2)--cycle;
\node (1) at (0,4)[rectangle]{$1$};
\node (2) at (1,2)[rectangle]{$2$};
\node (3) at (2,4){$3$};
\node (4) at (2,0)[rectangle]{$4$};
\node (5) at (3,2){$5$};
\node (6) at (4,4){$6$};
\node (7) at (6,4)[rectangle]{$7$};
\node (8) at (5,2)[rectangle]{$8$};
\node (9) at (4,0)[rectangle]{$9$};
\node (12) at (1,6)[rectangle,gray!30,fill=white]{12};
\node (11) at (3,6)[rectangle,gray!30,fill=white]{11};
\node (10) at (5,6)[rectangle,gray!30,fill=white]{10};
\node at (0,2)[draw=none]{$\cF_{in}$};
\node at (6,2)[draw=none]{$\cF_{out}$};
\node at (3,6.6)[draw=none, fill=none, gray!30]{$\cE$};
\drawpath{1,3,6,7}{red, thick}
\drawpath{2,5,8}{red, thick}
\drawpath{4,9}{red, thick}
\drawpath{8,6,5,3,2}{thick}
\drawpath{9,5,4}{thick}
\drawpath{4,2,1}{dashed}
\drawpath{7,8,9}{dashed}
\drawpath{7,10,6,11,3,12,1}{thick, gray, opacity=0.3}
\drawpath{12,11,10}{dashed,gray, opacity=0.3}
\end{tikzpicture}
\caption{Type $A_3$ with $\bi=(1,2,1,3,2,1)$}\label{A3q}
\end{subfigure}

 \begin{subfigure}[H]{\textwidth}
\centering
\begin{tikzpicture}[scale=0.9, every node/.style={inner sep=0, minimum size=0.4cm, thick, fill=white, draw,circle}, x=1.2cm, y=1cm]
\draw [gray!70](-1.2,6.1)--(7.2,6.1)--(3,-2.2)--cycle;
\node (1) at (0,4)[rectangle]{$1$};
\node (2) at (1,2)[rectangle]{$2$};
\node (3) at (2,4){$3$};
\node (4) at (2,2){$4$};
\node (5) at (2,0)[rectangle]{$5$};
\node (6) at (3,2){$6$};
\node (7) at (4,4)[]{$7$};
\node (8) at (4,2)[]{$8$};
\node (9) at (3,0)[]{$9$};
\node (10) at (6,4)[rectangle]{$10$};
\node (11) at (5,2)[rectangle]{$11$};
\node (12) at (4,0)[rectangle]{$12$};
\node (13) at (1,6)[rectangle,gray!30,fill=white]{13};
\node (14) at (3,6)[rectangle,gray!30,fill=white]{14};
\node (15) at (5,6)[rectangle,gray!30,fill=white]{15};
\node at (0,2)[draw=none]{$\cF_{in}$};
\node at (6,2)[draw=none]{$\cF_{out}$};
\node at (3,6.6)[draw=none, fill=none, gray!30]{$\cE$};
\drawpath{3,13,1}{thin, gray!30}
\drawpath{6,14,4}{vthick, gray!30}
\drawpath{12,15,9}{vthick, gray!30}
\drawpath{1,3,7,10}{red, thin}
\drawpath{10,8,7,4,3,2}{vthick}
\drawpath{2,4,6,8,11}{red,vthick}
\drawpath{11,9,6,5}{vthick}
\drawpath{5,9,12}{red,vthick}
\drawpath{5,2,1}{dashed, vthick}
\drawpath{12,11,10}{dashed, vthick}
\drawpath{13,14,15}{dashed,vthick, gray, opacity=0.3}
\end{tikzpicture}
\caption{Type $B_3$ with $\bi=(1,2,1,2,3,2,1,2,3)$}\label{B3q}
\end{subfigure}
\caption{The quiver $Q_F^\bi\subset Q^\bi$ attached to a triangle, and the $F_i$-paths highlighted in red.}\label{AB3q}
\end{figure}
\section{Cluster realization of quantum groups}\label{sec:realization}
Taking $Q:=Q^{\bi}$ as the basic quiver, we have an embedding of $\cU_q(\g)$ into the quantum torus algebra $\cX^{\over[\bi]\bi}$ (where the subquiver $\over[Q]:=Q^{\over[\bi]}$ is a mirror image of $Q^{\bi}$). In particular, the restriction to the Borel part $\cU_{q\til{q}}(\fb_\R)$ generated by $\{F_i, K_i\inv\}_{i\in \cI}$ only requires the subalgebra $\cX_F^{\over[\bi]\bi}$. Following \cite{Ip16}, we will use
\Eq{
K_i':=K_i\inv
}
instead for typesetting purpose, which originally refers to the opposite Cartan generators of the Drinfeld's double.

Let $\bi=(i_1,...,i_N)$ and $\over[\bi]=(\over[i]_1,...,\over[i]_N):=(i_N,...,i_1)$. Let
$I\simeq \over[I]\simeq \{1,..., N+2n\}$ be the corresponding vertices of $Q^{\bi}, Q^{\over[\bi]}$, and by abuse of notation we write $$\over[i_k]:=\over[i]_{\over[k]}.$$ Let us define a bijection $\s:\over[I]\to I, \over[k]\mapsto l$ such that $\over[i_k]=i_l$ and 
\Eq{
\left|\{j:i_j=i_l, j<l\}\right| = \left|\{\over[j]: \over[i_j]=\over[i_k], \over[j]<\over[k]\}\right|,
}
i.e. the indices appear in the same horizontal position in the quiver $Q^\bi, Q^{\over[\bi]}$ from the left.

Recall the expression of the positive representation $\cP_\l$ given in Theorem \ref{FKaction}. If we redefine $$\over[\bf]^{\over[k],-} := \bf^{\s(k),-}$$ and rewrite the sum as
\Eq{
\bf_i&=\bf_i^-+\bf_i^+\\
&=\sum_{\over[k]:\over[i_k]=i}\over[\bf]^{\over[k],-}+\sum_{k:i_k=i} \bf^{k,+},\nonumber
}
then we observe that each monomial will $q_i^{-2}$ commute with all the terms to the right. 

Define the consecutive ratios of the $\bf^{k,\pm}$ monomials as
\Eq{\label{XfromF}
X_{\over[k]}&=\case{\over[\bf]^{\over[k],-}& \over[k]\in \over[\cF]_{in},\\
q_i\inv \over[\bf]^{\over[k],-}(\over[\bf]^{\over[k]^-,-})\inv& \over[k]\notin \over[\cF]_{in}\cup \over[\cF]_{out},}\\\label{XfromF2}
X_{k}&=\case{q_i\inv \bf^{k,+}(\over[\bf]^{\over[k]^*,-})\inv&k\in \cF_{in},\\
q_i\inv \bf^{k,+}(\bf^{k^-,+})\inv& k\notin \cF_{in}\cup \cF_{out},\\
q_i\inv K_{i_k}\inv (\bf^{k^*,+})\inv&k\in \cF_{out}.}
}

Let $\cP_\l^\fb$ be the positive representations $\cP_\l$ restricted to the Borel part $\cU_{q\til{q}}(\fb_\R)$. Then from the explicit expression \eqref{FF} of the operators and the commutation relations from Definition \ref{basic}, we have
\begin{Thm}\label{Fmain}\cite{Ip16} The assignment \eqref{XfromF}-\eqref{XfromF2} gives a positive representation of $\cX_F^{\over[\bi]\bi}$ generated by $\{X_{\over[k]},X_k\}_{k=1}^{N+n}$ in the sense of Definition \ref{posrepX}.

Equivalently, let $\cP_{F_i}, \over[\cP]_{F_i}$ be the $F_i$ paths of $Q^{\bi}$ and $Q^{\over[\bi]}$ respectively. Let
\Eq{
\bf_i&:=X(\over[\cP]_{F_i}\cP_{F_i})\in \cX_F^{\over[\bi]\bi},\\
K_i'&:=X_{\over[\cP]_{F_i}\cP_{F_i}}\in \cX_F^{\over[\bi]\bi}\nonumber
}
as in Definition \ref{path}. Then the assignment \eqref{XfromF}-\eqref{XfromF2} coincides with the positive representation $\cP_\l^\fb$ of the Borel part. We also have $$\bf_i^- = X(\over[\cP]_{F_i}).$$

For the tensor product $\cP_\l^\fb\ox \cP_\mu^\fb$, the coproduct $\D(\bf_i), \D(K_i')$ are represented on $\cX^{\over[\bi]\bi}\ox \cX^{\over[\bi]\bi}$ in a similar way by concatenation of the $F_i$-paths of two copies of the quiver $Q^{\over[\bi]\bi}, {Q'}^{\over[\bi]\bi}$ amalgamated along $\cF_{out}=\cF_{in}'$.
\end{Thm}

On the other hand, we observe in \cite{Ip14} that the Borel part can alternatively be represented using just half of the sum given by
\Eq{\label{feigin}\bf_i=\sum_{k:i_k=i}\bf^{k,+},}
 which coincides with the so-called \emph{Feigin's homomorphism}. Let us call this representation $\til{\cP_\l^\fb}$ (which we called it the \emph{standard form} in \cite{Ip14}).
\begin{Lem}\label{indep} $\til{\cP_\l^\fb}\simeq\til{\cP^\fb}$ is independent of the parameter $\l$.
\end{Lem}
\begin{proof}The proof is similar but easier than the one in \cite{Ip14}. From the explicit expression for $K_i$, since the Cartan matrix $A=(a_{ij})$ is invertible, we can always make a shift of variables $$u_k\mapsto u_k-c_k$$ for some constant $c_k$, where $k\in \cF_{in}$, such that the parameter $\l_i$ vanishes in the expression of $K_i$, while each monomial term $\bf^{k,+}$ picks up some new parameters $c_k'$. Now since each $\bf^{k,+}$ has a unique momentum operator $2p_k$, by the unitary transformation given by multiplication by $e^{\pi i u_k c_k'}$ for some constant $c_k'$, which acts as $$2p_k\mapsto 2p_k-c_k',$$ we can make the resulting parameters $c_k'$ vanish.
\end{proof}

Modifying \eqref{XfromF2} slightly with
\Eq{
\label{XfromF3}
X_{k}&=\case{\bf^{k,+}&k\in \cF_{in},\\
q_i\inv \bf^{k,+}(\bf^{k^-,+})\inv& k\notin \cF_{in}\cup \cF_{out},\\
q_i\inv K_i\inv (\bf^{k^*,+})\inv&k\in \cF_{out},}
}
we conclude that
\begin{Thm} The assignment \eqref{XfromF3} gives a positive representation of the quantum torus algebra $\cX_F^\bi\subset \cX^\bi$ associated to $Q_F^{\bi}$. Equivalently, let
\Eq{\bf_i&:=X(\cP_{F_i})\in \cX_F^\bi,\\
 K_i'&:= X_{\cP_{F_i}}\in \cX_F^\bi,\nonumber
 }
 then the assignment \eqref{XfromF3} gives the representation $\til{\cP^\fb}$ of $\cU_{q\til{q}}(\fb_\R)$.
\end{Thm}

We also have the notion of an \emph{$E_i$-path}. Consider the quiver $Q_{D_{2,1}}^{\over[\bi]\bi}$ associated to a once punctured disk with two marked points, obtained by gluing $Q^{\over[\bi]}, Q^{\bi}$ along $\over[\cF]_{out}=\cF_{in}$ as well as $\over[\cE]=\cE$ (see Figure \ref{PP}). Recall that due to Theorem \ref{FKaction}, the generators $\be_i$ can be represented as a sum of monomials
$$\be_i=\be_i^-+\be_i^+.$$
In particular, if $\bi$ is well-chosen, $\be_i$ can be represented again by a path polynomial (with slight modification in type $C_n, E_8, F_4$ and $G_2$, see \cite{Ip16}) of the form
\Eq{
\be_i=X(\cP_{E_i}\over[\cP]_{E_i})\in \cX_{D_{2,1}}^{\over[\bi]\bi},
}
where $\cP_{E_i}$ and $\over[\cP]_{E_i}$ are respectively a path on the $Q^{\bi}, Q^{\over[\bi]}$ quivers, called the \emph{$E_i$-paths}. The path $\cP_{E_i}$ starts at $\cF_{out}$ and ends at $\cE$, while $\over[\cP]_{E_i}$ starts at $\over[\cE]$ and ends at $\over[\cF]_{in}$. Similarly we have $$\be_i^-=X(\cP_{E_i}).$$
\section{Main results}\label{sec:flip}
We are now ready to give an alternative proof of the main theorem in \cite{Ip14}.
\begin{Thm}\label{main}
We have the following unitary equivalences of positive representations restricted to the Borel part $\cU_{q\til{q}}(\fb_\R)$:
\Eq{
\cP_\l^\fb&\simeq \til{\cP^\fb},\\
\til{\cP^\fb}\ox \til{\cP^\fb}&\simeq \til{\cP^\fb}\ox \cM,
}
for some multiplicity module $\cM\simeq L^2(\R^N)$ where $\cU_{q\til{q}}(\fb_\R)$ acts trivially. This means that the positive representations restricted to $\cU_{q\til{q}}(\fb_\R)$ is closed under taking tensor product.
\end{Thm}

Recall that given an embedding of 
$$\cU_q(\g)\inj \cX_{D_{2,2}}^{\over[\bi]\bi\over[\bi]\bi}\subset \cX_{D_{2,1}}^{\over[\bi]\bi}\ox\cX_{D_{2,1}}^{\over[\bi]\bi}$$
into the quantum torus algebra associated to a disk with 2 punctures and 2 marked points given by the coproduct, we show in \cite{Ip16} that the reduced $\cR$ operator can be decomposed into
\Eq{
\over[\cR]=\cR_4\cdot \cR_3\cdot \cR_2\cdot \cR_1,
}
where
\Eqn{
R_4&=g_{b_{i_N}}(\be_{i_N}^+\ox \bf^{N,+})... g_{b_{i_2}}(\be_{i_2}^+\ox \bf^{2,+})g_{b_{i_1}}(\be_{i_1}^+\ox \bf^{1,+}),\\
R_3&=g_{b_{i_N}}(\be_{i_N}^-\ox \bf^{N,+})... g_{b_{i_2}}(\be_{i_2}^-\ox \bf^{2,+})g_{b_{i_1}}(\be_{i_1}^-\ox \bf^{1,+}),\\
R_2&=g_{b_{i_1}}(\be_{i_1}^+\ox \bf^{1,-})g_{b_{i_2}}(\be_{i_2}^+\ox \bf^{2,-})... g_{b_{i_N}}(\be_{i_N}^+\ox \bf^{N,-}),\\
R_1&=g_{b_{i_1}}(\be_{i_1}^-\ox \bf^{1,-})g_{b_{i_2}}(\be_{i_2}^-\ox \bf^{2,-})... g_{b_{i_N}}(\be_{i_N}^-\ox \bf^{N,-}).
}
Also $\cR_2$ commute with $\cR_3$. By the correspondence in Lemma \ref{useful}, the 4 factors correspond to sequences of quiver mutations associated to flipping of triangulation in different configurations. In this paper, we show explicitly the following fact, which is only implied implicitly in \cite{Ip16}. 
\begin{Thm} The quiver mutations
 $$\mu_{\cR_k}^q:=\mu_{m_1}^q\cdots\mu_{m_M}^q$$
 induced by $\cR_k$, $k=1,3$, corresponding to the flip of triangulation preserve the $F_i$-paths (see Figure \ref{flip}). More precisely, for an amalgamation $Q^{\bi\bi'}$ of two quivers with $\bi'=\over[\bi]$ $(k=1)$ or $\bi'=\bi$ $(k=3)$ along $\cF_{out}=\cF_{in}'$, if $\cP_{F_i}, \cP_{F_i}'$ are the $F_i$ paths of the left and right triangle respectively, and $\what{\cP}_{F_i}$ are the $F_i$ paths of the bottom triangle of the mutated quiver, then
\Eq{\mu_{\cR_k}^q(\what{X}(\what{\cP}_{F_i})) =(X(\cP_{F_i}\cP_{F_i}')),\tab k=1,3,
}
where $\what{X}_j$ belongs to the quantum torus algebra $\what{\cX}^{\bi\bi'}$ of the mutated quiver, amalgamated along $\what{\cE}=\what{\cF}_{out}'$ (see Figure \ref{flip}).
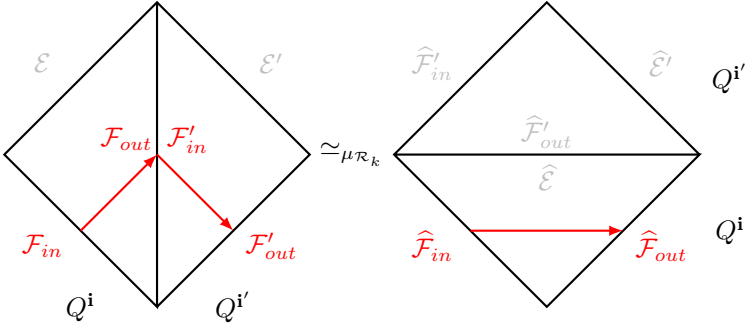
\begin{figure}[!htb]
\centering
\begin{tikzpicture}[baseline=(0),x=2cm,y=2cm]
\draw[thick] (0,0)--(1,1)--(0,2)--(-1,1)--(0,0)--(0,2);
\node (0) at (0,1){};
\node at (-0.75,0.4)[red]{$\cF_{in}$};
\node at (-0.2,1.1)[red]{$\cF_{out}$};
\node at (0.2,1.1)[red]{$\cF_{in}'$};
\node at (0.75,0.4)[red]{$\cF_{out}'$};
\node at (0.75,1.6)[gray!50]{$\cE'$};
\node at (-0.75,1.6)[gray!50]{$\cE$};
\path (-0.5,0.5) edge[->, thick, red] (0,1);
\path (0,1) edge[->, thick, red] (0.5,0.5);
\node at (-0.5, 0){$Q^\bi$};
\node at (0.5, 0){$Q^{\bi'}$};
\end{tikzpicture}
$\simeq_{\mu_{\cR_k}}$
\begin{tikzpicture}[baseline=(0),x=2cm,y=2cm]
\draw[thick] (1,1)--(0,2)--(-1,1)--(0,0)--(1,1)--(-1,1);
\node (0) at (0,1){};
\node at (0.75,1.6)[gray!50]{$\what{{\cE}}'$};
\node at (0,1.15)[gray!50]{$\what{{\cF}}'_{out}$};
\node at (0,0.85)[gray!50]{$\what{\cE}$};
\node at (-0.75,1.6)[gray!50]{$\what{{\cF}}'_{in}$};
\node at (-0.75,0.4)[red]{$\what{\cF}_{in}$};
\node at (0.75,0.4)[red]{$\what{\cF}_{out}$};
\path (-0.5,0.5) edge[->, thick, red] (0.5,0.5);
\node at (1.2, 0.5){$Q^\bi$};
\node at (1.2, 1.5){$Q^{\bi'}$};
\end{tikzpicture}
\caption{Configurations of the mutations $\mu_{\cR_k}$, $k=1,3$, with the $F_i$ paths schematically shown in red, and the irrelevant extremal nodes grayed out.}\label{flip}
\end{figure}
\end{Thm}

According to the result of \cite{Ip16}, the quiver mutations corresponding to $\cR_k$ using Lemma \ref{useful} are of the type as in Figure \ref{type1}, and in the doubly-laced case also as in Figure \ref{type2}. (The case of type $G_2$ was treated explicitly in \cite{Ip16} already and will be omitted.)

\begin{figure}[H]
\centering
\begin{tikzpicture}[baseline=(0), every node/.style={inner sep=0, minimum size=0.5cm, thick}, x=0.3cm, y=0.5cm]
\node (0) at (0,1.2){};
\node at (-3,3){$\cdots$};
\node at (15,3){$\cdots$};
\node at (-3,0){$\cdots$};
\node at (15,0){$\cdots$};
\node at (0,1.7){$\vdots$};
\node at (12,1.7){$\vdots$};
\node (1) at (0,3) [draw, circle] {$k^-$};
\node (2) at (6,3) [draw, circle] {$k$};
\node (3) at (12,3) [draw, circle] {$k^+$};
\node (4) at (0, 0) [draw, circle]{$l$};
\node (5) at (12,0) [draw,circle]{$l^+$};
\path (1) edge[->, thick] (2);
\path (2) edge[->, thick] (3);
\path (2) edge[->, thick] (4);
\path (5) edge[->, thick] (2);
\path (4) edge[->, thick] (5);
\end{tikzpicture}
$\simeq_{\mu_k}$
\begin{tikzpicture}[baseline=(0), every node/.style={inner sep=0, minimum size=0.5cm, thick}, x=0.3cm, y=0.5cm]
\node at (-3,3){$\cdots$};
\node at (15,3){$\cdots$};
\node at (-3,0){$\cdots$};
\node at (15,0){$\cdots$};
\node at (0,1.7){$\vdots$};
\node at (12,1.7){$\vdots$};
\node (0) at (0,1.2){};
\node (1) at (0,0) [draw, circle] {$l$};
\node (2) at (6,0) [draw, circle] {$k$};
\node (3) at (12,0) [draw, circle] {$l^+$};
\node (4) at (0, 3) [draw, circle]{$k^-$};
\node (5) at (12,3) [draw,circle]{$k^+$};
\path (1) edge[->, thick] (2);
\path (2) edge[->, thick] (3);
\path (2) edge[->, thick] (4);
\path (5) edge[->, thick] (2);
\path (4) edge[->, thick] (5);
\end{tikzpicture}
\caption{Mutation at simply-laced part.}\label{type1}
\end{figure}
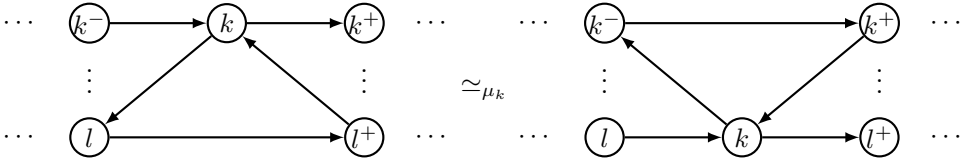
\begin{figure}[H]
\centering
\begin{tikzpicture}[baseline=(0), every node/.style={inner sep=0, minimum size=0.5cm, thick}, x=0.28cm, y=0.5cm]
\node at (-3,3){$\cdots$};
\node at (15,3){$\cdots$};
\node at (-3,0){$\cdots$};
\node at (15,0){$\cdots$};
\node at (0,1.7){$\vdots$};
\node at (12,1.7){$\vdots$};
\node (0) at (0,1.2){};
\node (1) at (0,3) [draw, circle] {$k^-$};
\node (2) at (6,3) [draw, circle] {$k$};
\node (3) at (12,3) [draw, circle] {$k^+$};
\node (4) at (0, 0) [draw, circle]{$l^-$};
\node (5) at (6,0) [draw,circle]{$l$};
\node (6) at (12,0) [draw,circle]{$l^+$};
\path (1) edge[->, vthick] (2);
\path (2) edge[->, vthick] (3);
\path (4) edge[->, thin] (5);
\path (5) edge[->, thin] (6);
\path (2) edge[->, vthick] (4);
\path (5) edge[->, vthick] (2);
\path (3) edge[->, vthick] (5);
\end{tikzpicture}
$\simeq_{\mu_l\mu_k\mu_l}$
\begin{tikzpicture}[baseline=(0), every node/.style={inner sep=0, minimum size=0.5cm, thick}, x=0.28cm, y=0.5cm]
\node at (-3,3){$\cdots$};
\node at (15,3){$\cdots$};
\node at (-3,0){$\cdots$};
\node at (15,0){$\cdots$};
\node at (0,1.7){$\vdots$};
\node at (12,1.7){$\vdots$};
\node (0) at (0,1.2){};
\node (1) at (0,3) [draw, circle] {$k^-$};
\node (2) at (6,3) [draw, circle] {$k$};
\node (3) at (12,3) [draw, circle] {$k^+$};
\node (4) at (0, 0) [draw, circle]{$l^-$};
\node (5) at (6,0) [draw,circle]{$l$};
\node (6) at (12,0) [draw,circle]{$l^+$};
\path (1) edge[->, vthick] (2);
\path (4) edge[->, thin] (5);
\path (5) edge[->, thin] (6);
\path (6) edge[->, vthick] (2);
\path (2) edge[->, vthick] (5);
\path (5) edge[->, vthick] (1);
\path (2) edge[->, vthick] (3);
\end{tikzpicture}
\caption{Mutation at doubly-laced part.}\label{type2}
\end{figure}
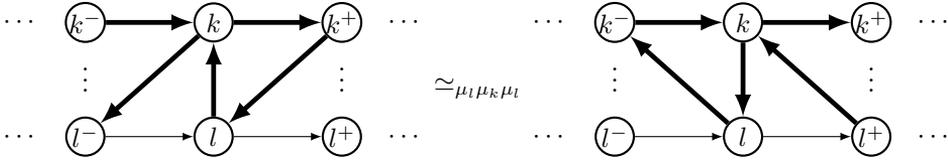
In particular, this means that the mutation of an $F_i$-path is again a path and its direction is preserved: in the simply-laced moves, the path polynomials are given by
\Eqn{
\mu_k^q(...+\what{X}_{..., k^-}+\what{X}_{..., k^-,k^+}+...)=(...+X_{..., k^-}+X_{..., k^-,k}+X_{..., k^-,k,k^+}+...),
}
while the path polynomials are preserved in the doubly-laced moves.

In terms of unitary equivalence, we have
\begin{Prop}\label{ii'}
Let $\bi'=\over[\bi]$ and consider $Q^{\bi\over[\bi]}$ as in Figure \ref{flip}. Let also $\cP_{E_i}$ be the $E_i$-paths for the left part $Q^{\bi}$ of the quiver, $\cP_{F_i}, \over[\cP]_{F_i}$ be the $F_i$-paths for the left part $Q^{\bi}$ and right part $Q^{\over[\bi]}$ of the quiver respectively. Let  
\Eqn{
X(\cP_{E_i})&=:\underline{\be}_i^-\ox1,\\
X(\over[\cP]_{F_i}) &=:1\ox \underline{\bf}_i^-\\
&=: 1\ox\sum_{k:i_k=i} \underline{\bf}^{k,-}
}
and define
\Eq{\label{phi1}\Phi_1:=g_{b_{i_1}}(\underline{\be}_{i_1}^-\ox \underline{\bf}^{1,-})g_{b_{i_2}}(\underline{\be}_{i_2}^-\ox \underline{\bf}^{2,-})... g_{b_{i_N}}(\underline{\be}_{i_N}^-\ox \underline{\bf}^{N,-}).
}
Then we have
\Eq{Ad_{\Phi_1} (X(\cP_{F_i}\over[\cP]_{F_i})) &= X(\cP_{F_i})=M_{\cR_1}(\what{X}(\what{\cP}_{F_i}))\in \cX_\bi\ox 1\\
Ad_{\Phi_1} (X_{\cP_{F_i}\over[\cP]_{F_i}}) &= X_{\cP_{F_i}\over[\cP]_{F_i}}=M_{\cR_1}(\what{X}_{\what{\cP}_{F_i}})\in \cX_\bi\ox \cX_{\over[\bi]}
}
where $M_{\cR_1}=\mu_{m_1}'\cdots \mu_{m_M}'$ is the monomial transform $\what{\cX}_{\bi\over[\bi]}\to \cX_{\bi\over[\bi]}$.
\end{Prop}
\begin{proof}Let also 
\Eqn{
X(\cP_{F_i}) &=: \underline{\bf}_i^+\ox 1,\\
X_{\cP_{F_i}}&=:\underline{K}_i'\ox1.
}
Then the commutation relations among $\{\underline{\be}_i^-,\underline{\bf}_i^+,\underline{\bf}^{k,-},\underline{K}'_i\}$ is exactly the same as $\{\be_i^-,\bf_i^+, \bf^{k,-}, K_i'\}$ from Theorem \ref{FKaction}. Hence by Proposition \ref{efcom} and \eqref{g12}, we have 
\Eqn{
Ad_{g_{b_{i_k}}(\be_{i_k}\ox \bf^{k,-})} (\bf_i^+\ox 1)&= 
\case{\bf_i^+\ox 1+K_i'\ox \bf^{k,-}& i_k = i,\\
\bf_i^+\ox 1 &i_k\neq i,}\\
Ad_{g_{b_{i_k}}(\be_{i_k}\ox \bf^{k,-})}(K_i' \ox \bf^{l,-}) &= K_i' \ox \bf^{l,-}, \tab k<l,\\
Ad_{g_{b_{i_k}}(\be_{i_k}\ox \bf^{k,-})} (K_i'\ox K_i') &= K_i'\ox K_i'.
}
Since $$X(\cP_{F_i}\over[\cP]_{F_i})=\underline{\bf}_i^+\ox 1+\underline{K}_i' \ox \underline{\bf}_i^-$$ and $$X_{\cP_{F_i}\over[\cP]_{F_i}}=\underline{K}_i'\ox \underline{K}_i',$$ by induction we have the first equality.
Since the mutated $F_i$ paths are the unique paths joining $\cF_{in}$ and $\cF_{out}$ of the bottom triangle, the $F_i$-paths are preserved under mutation. By Lemma \ref{useful} we have the claim for the monomial transform.
\end{proof}
Note that the embedding to the quiver $Q^{\bi\over[\bi]}$ gives the positive representation $\cP_\l^{\over[\bi]}$ restricted to the Borel part $\cU_{q\til{q}}(\fb_\R)$ for the word $\over[\bi]$ instead, but $$\cP_\l\simeq\cP_\l^{\over[\bi]}\simeq \cP_\l^{\bi}.$$ Hence applying Proposition \ref{ii'} we obtain the unitary transformation of the positive self-adjoint operators $\bf_i$ of $\cP_\l^\fb$ to $\til{\cP_\l^\fb}\simeq \til{\cP^\fb}$:
\Eq{
\bf_i=\bf_i^++\bf_i^-\simeq \bf_i^+.
} 
It remains to show that $K_i'$ depends only on $Q^{\bi}$ for a unitary transformation that preserves $\bf_i^+$:
\begin{Prop}\label{K1}As positive self-adjoint operators, we have the unitary equivalence
\Eq{
K_i'=X_{\cP_{F_i}\over[\cP]_{F_i}}\simeq X_{\cP_{F_i}}
}
\end{Prop}
\begin{proof} The factors in the expression of $K_i'$ involving the second path $\over[\cP]_{F_i}$ commute with the first factor, hence we can just treat them as parameters of the representation. Hence using exactly the same argument as in Lemma \ref{indep}, we arrive at the conclusion.
\end{proof}
This proves the first claim of Theorem \ref{main}.

In exactly the same way, for $Q^{\bi\bi}$ as in Figure \ref{flip} for $\bi'=\bi$, we have
\begin{Prop}
Let $\cP_{E_i}$ be the $E_i$-paths for the left part $Q_1^{\bi}$ of the quiver, and $\cP_{F_i}^1,\cP_{F_i}^2$ be the $F_i$-paths for the left part $Q_1^{\bi}$ and the right part $Q_2^{\bi}$ of the quiver respectively. Let 
\Eqn{
X(\cP_{E_i})&=:\underline{\be}_i^-\ox1,\\
 X(\cP^2_{F_i}) &=: 1\ox\sum_{k:i_k=i} \underline{\bf}^{k,+},
 } and define
\Eq{\label{phi3}\Phi_3=g_{b_{i_N}}(\underline{\be}_{i_N}^-\ox \underline{\bf}^{N,+})... g_{b_{i_2}}(\underline{\be}_{i_2}^-\ox \underline{\bf}^{2,+})g_{b_{i_1}}(\underline{\be}_{i_1}^-\ox \underline{\bf}^{1,+}).
}
Then
\Eq{Ad_{\Phi_3} (X(\cP^1_{F_i}\cP^2_{F_i})) &= X(\cP^1_{F_i})\in \cX_\bi\ox 1,\\
Ad_{\Phi_3} (X_{\cP^1_{F_i}\cP^2_{F_i}}) &= X_{\cP^1_{F_i}\cP^2_{F_i}}\in \cX_\bi\ox \cX_{\bi}.
}
\end{Prop}

Since the coproduct on $\til{\cP^\fb}\ox \til{\cP^\fb}$ is simply represented by concatenation of the $F_i$-paths, together with exactly the same argument as in Proposition \ref{K1} for the $K_i'$ generators, we conclude that
\Eqn{
\D(\bf_i)&\simeq \bf_i\ox 1,\\
\D(K_i')&\simeq K_i'\ox 1,
}
and hence
$$\til{\cP^\fb}\ox \til{\cP^\fb} \simeq \til{\cP^\fb}\ox \cM$$
for a multiplicity module $\cM\simeq L^2(\R^N)$, and this concludes the proof of the main theorem. 

\hfill \ensuremath{\Box}

In particular, since one can decompose $g_{b_i}(\be_i^-)=\prod g_{b_{...}}(X_{...})$ into product of quantum dilogarithms with quantum cluster monomials as arguments \cite{Ip16}, one can rewrite the unitary transformation $\Phi_k$ as honest unitary operators acting on $L^2(\R^N)$ easily using the formula \eqref{XfromF}-\eqref{XfromF2} to rewrite the quantum cluster variables $X_i$ as positive self-adjoint operators in terms of the canonical variables $\{u_i, p_i\}$. Therefore this completely solves the combinatorial problem posed in \cite{Ip14} for explicitly writing down the tensor product decomposition of positive representations restricted to $\cU_{q\til{q}}(\fb_\R)$ of all types other than type $A_n$, which cannot be done previously without the cluster realization.
\section{Examples}\label{sec:ex}
We consider several examples to illustrate the construction above. We will describe the basic quiver, and the positive representations realized on the quiver $Q^{\bi\over[\bi]}$. We recall the $\cP_{E_i}$ paths found in \cite{Ip16} and its quantum dilogarithm decomposition. This allows us to describe the unitary transformation $\Phi_1$ and $\Phi_3$ explicitly using \eqref{phi1}, \eqref{phi3} corresponding to certain sequence $\mu$ of quiver mutations on $Q^{\bi\over[\bi]}$ and $Q^{\bi\bi}$ respectively. We will use the standard indexing, and display the $Q_F^{\bi\bi'}$ quiver and its mutation, highlighting the $\cP_{F_i}$ paths, while grayed out the irrelevant parts incident to the frozen nodes $\cE$. 

\begin{Rem}For the first flip, in the more general setting corresponding to self-folded triangles where the frozen nodes in $\cE$ of both triangles are identified (see Figure \ref{PP}), although the $Q_F^{\bi\over[\bi]}$ part will not be affected, the grayed out part of the resulting mutated quiver after $\mu_{\cR_1}$ is much more complicated and interesting (see \cite{SS17} for an example in type $A_n$), and is considered an important combinatorial part for the general problem of tensor product decomposition of positive representations of the whole quantum group $\cU_{q\til{q}}(\g_\R)$.
\end{Rem}
\subsection{Type $A_1$}\label{sec:a1}
The representation theory of $\cU_{q\til{q}}(\sl(2,\R))$, first described by Faddeev \cite{Fa2}, is studied extensively in various papers \cite{BT,  Ip1, PT1, PT2}. In terms of the notations in this paper, it reads the following. The basic quiver is simply given by 3 nodes as in Figure \ref{basicA1}. Obviously $Q^{\bi}\simeq Q^{\over[\bi]}$, but let us still denote the nodes of $Q^{\over[\bi]}$ by $\over[k]\in\over[I]$.
\begin{figure}[!htb]
\centering
\begin{tikzpicture}[baseline=(3),every node/.style={inner sep=0, minimum size=0.5cm, thick}, x=0.6cm, y=0.6cm]
\node (1) at (-2,0)[draw]{$1$};
\node (2) at (2,0)[draw]{$2$};
\node (3) at (0,2)[draw,gray!30]{$3$};
\node at (-3.1,0){$\cF_{in}$};
\node at (3.3,0){$\cF_{out}$};
\node at (0,3)[gray!30]{$\cE$};
\drawpath{1,2}{thick};
\drawpath{2,3,1}{gray!30, thick};
\end{tikzpicture}
\caption{Basic quiver $Q_{A_1}^{\bi}$.}\label{basicA1}
\end{figure}

The positive representation $\cP_\l^\fb$ is represented through $\cX^{\bi\over[\bi]}$ (where $\over[1]=2$) as
\Eqn{
\bf&=e^{\pi b(-u+2\l+2p)}+e^{\pi b(u-2\l+2p)}:=X_{1}+X_{1,\over[1]},\\
K'&=e^{2\pi b(u-\l)}:=X_{1,\over[1],\over[2]}.
}

\begin{figure}[!htb]
\centering
\begin{tikzpicture}[baseline=(0),every node/.style={inner sep=0, minimum size=0.5cm, thick}, x=0.7cm, y=0.7cm]
\node (0) at (0,1){};
\node (1) at (-4,0)[draw]{$1$};
\node (1') at (0,0)[draw, circle]{$\over[1]$};
\node (2') at (4,0)[draw]{$\over[2]$};
\node (3) at (-2,2)[draw, gray!30]{$3$};
\node (3') at (2,2)[draw, gray!30]{$\over[3]$};
\node at (-2,-1){$Q_{A_1}^{\bi}$};
\node at (2,-1){$Q_{A_1}^{\over[\bi]}$};
\drawpath{1,1',2'}{red,thick};
\drawpath{2',3',1',3,1}{gray!30, thick};
\end{tikzpicture}
$\simeq_{\mu_{\over[1]}}$
\begin{tikzpicture}[baseline=(2),every node/.style={inner sep=0, minimum size=0.5cm, thick}, x=0.7cm, y=0.7cm]
\node (1) at (-2,-2)[draw]{$1$};
\node (1') at (0,0)[draw, circle]{$\over[1]$};
\node (2') at (2,-2)[draw]{$\over[2]$};
\node (3) at (-2,2)[draw, gray!30]{$3$};
\node (3') at (2,2)[draw, gray!30]{$\over[3]$};
\node at (4,-1){$Q_{A_1}^{\bi}$};
\node at (4,1){$Q_{A_1}^{\over[\bi]}$};
\drawpath{1,2'}{red,thick}
\drawpath{2',1',1}{thick}
\drawpath{3',3,1',3'}{gray!30, thick}
\end{tikzpicture}
\caption{$\cP_\l^\fb\simeq \til{\cP^\fb}$ in type $A_1$ with the $F_i$ path shown in red.}\label{mutateA1}
\end{figure}

The $E$ and $F$-paths are $\cP_{E}=(2,3), \cP_{F}=(1,2)$ respectively, hence $\Phi_1$ is given by $$\Phi_1=g_b(\underline{\be}\ox \underline{\bf}^{1,-}) =g_b(X_2\ox X_{\over[1]}):= g_b(X_{\over[1]})=g_b(e^{\pi b(2u-4\l)}),$$ where we identified $X_{\over[1]}:=X_2\ox X_{\over[1]}$ in the embedding $\cX^{\bi\over[\bi]}\inj \cX^{\bi}\ox \cX^{\over[\bi]}$. It simply mutates at the node $\over[1]=2$ as in Figure \ref{mutateA1}. Using \eqref{gcon}, we obviously have
\Eqn{
Ad_{g_b(X_{\over[1]})}(\bf) &= Ad_{g_b(X_{\over[1]})}(X_1+X_{1,\over[1]}) = X_1 = M_{\cR_1}(\what{X}_1),\\
Ad_{g_b(X_{\over[1]})}(K') &= Ad_{g_b(X_{\over[1]})}(X_{1,\over[1],\over[2]}) = X_{1,\over[1],\over[2]} = M_{\cR_1}(\what{X}_{1,\over[2]}).
}
Finally, we can get rid of $\l$ by making a shift $$(2p\mapsto 2p-\l)\circ(u\mapsto u+\l)$$ to obtain the unitary equivalence $\cP_\l\simeq \til{\cP^\fb}$:
\Eqn{
\bf\simeq e^{\pi b(-u+2p)},\tab K'\simeq e^{2\pi b u}.
}
For the tensor product $\til{\cP^\fb}\ox\til{\cP^\fb}$ as in Figure \ref{finalA1}, we use $'$ to denote the nodes of the second tensor factor, such that 
\Eqn{
\D(\bf)&=e^{\pi(-u+2p)}+e^{\pi b (2u-u'+2p')}=X_1+X_{1,1'},\\
 \D(K')&=e^{2\pi b(u+u')}=X_{1,1',\over[2]'}.} Up to re-indexing, we have exactly the same transformation, where $$\Phi_3=g_b(X_{1'})=g_b(e^{\pi b(3u-u'-2p+2p')})$$ and
\Eqn{
Ad_{g_b(X_{1'})}(\D(\bf)) &= Ad_{g_b(X_{1'})}(X_1+X_{1,1'}) = X_1=M_{\cR_3}(\what{X}_1),\\
Ad_{g_b(X_{1'})}(\D(K')) &= Ad_{g_b(X_{1'})}(X_{1,1',\over[2]'})=X_{1,1',\over[2]'}=M_{\cR_3}(\what{X}_{1,\over[2]'}).
}
Again shifting by $$(2p\mapsto 2p-u')\circ (u\mapsto u-u')$$ we arrive at
\Eqn{
\D(\bf)\simeq e^{\pi(-u+2p)},\tab \D(K')\simeq e^{2\pi b u},
}
which gives $$\til{\cP^\fb}\ox \til{\cP^\fb}\simeq \til{\cP^\fb}\ox \cM$$ as required.

\begin{figure}[!htb]
\centering
\begin{tikzpicture}[baseline=(0),every node/.style={inner sep=0, minimum size=0.5cm, thick}, x=0.7cm, y=0.7cm]
\node (0) at (0,1){};
\node (1) at (-4,0)[draw]{$1$};
\node (1b) at (-2,2)[draw, gray!30,circle]{$\over[1]$};
\node (1') at (0,0)[draw,circle, red]{$1'$};
\node (1b') at (2,2)[draw, gray!30,circle]{$\over[1]'$};
\node (2b') at (4,0)[draw]{$\over[2]'$};
\node at (-2,-1){$Q_{A_1}^{\bi}$};
\node at (2,-1){$Q_{A_1}^{\bi}$};
\drawpath{1,1',2b'}{red,thick}
\drawpath{2b',1b',1',1b,1}{gray!30,thick}
\end{tikzpicture}
$\simeq_{\mu_{1'}}$
\begin{tikzpicture}[baseline=(1'),every node/.style={inner sep=0, minimum size=0.5cm, thick}, x=0.7cm, y=0.7cm]
\node (1) at (-2,-2)[draw]{$1$};
\node (1b) at (-2,2)[draw, circle,gray!30]{$\over[1]$};
\node (1') at (0,0)[draw, circle]{$1'$};
\node (1b') at (2,2)[draw, circle,gray!30]{$\over[1]'$};
\node (2b') at (2,-2)[draw]{$\over[2]'$};
\node at (4,-1){$Q_{A_1}^{\bi}$};
\node at (4,1){$Q_{A_1}^{\bi}$};
\drawpath{1,2b'}{red,thick}
\drawpath{2b',1',1}{thick}
\drawpath{1',1b',1b,1'}{gray!30, thick}
\end{tikzpicture}
\caption{$\til{\cP^\fb}\ox \til{\cP^\fb}\simeq \til{\cP^\fb}\ox \cM$ in type $A_1$.}\label{finalA1}
\end{figure}
\subsection{Type $A_n$}\label{sec:an}
Let $$\bi=(1,2,1,3,2,1,...,n,n-1,...,1)$$ be the standard word. We note that $\over[\bi]$ differs from $\bi$ only by the moves $(i,j)\corr(j,i)$ for $|i-j|>1$. In particular, the quiver $Q^{\bi}\simeq Q^{\over[\bi]}$ is identical up to re-indexing (in fact it establishes a $\Z_3$ symmetry). Therefore the mutation sequence corresponding to $\cR_1$ and $\cR_3$ is again identical up to re-indexing. Two different but equivalent mutation sequences for $\cR_k$ has been established explicitly in \cite{Ip16} and \cite{SS16}. Let us apply this for our running example in the case $A_3$.

The basic quiver for $A_3$ is drawn in Figure \ref{A3q}. The positive representation $\cP_\l^\fb$ is given on $\cX^{\bi\over[\bi]}$ (where $\over[1]=7, \over[2]=8, \over[3]=9$) as shown in the left quiver of Figure \ref{mutateA3} by
\Eqn{
\bf_1&=X(1,3,6,\over[1],\over[4],\over[6]),& K_1' &=X_{1,3,6,\over[1],\over[4],\over[6],\over[7]},\\
\bf_2&=X(2,5,\over[2],\over[5]), &K_2' &=X_{2,5,\over[2],\over[5],\over[8]},\\
\bf_3&=X(4,\over[3]),& K_3' &=X_{4,\over[3],\over[9]}.
}

\begin{figure}[H]
\centering
\begin{tikzpicture}[baseline=(3), every node/.style={inner sep=0, minimum size=0.4cm, thick, fill=white, draw,circle}, x=0.8cm, y=0.8cm]
\node (1) at (-3,-1)[rectangle]{$1$};
\node (2) at (-2,-2)[rectangle]{$2$};
\node (3) at (-2,0){$3$};
\node (4) at (-1,-3)[rectangle]{$4$};
\node (5) at (-1,-1){$5$};
\node (6) at (-1,1){$6$};
\node (1') at (0,2){$\over[1]$};
\node (2') at (0,0){$\over[2]$};
\node (3') at (0,-2){$\over[3]$};
\node (10) at (-1,3)[rectangle,gray!30,fill=white]{10};
\node (11) at (-2,2)[rectangle,gray!30,fill=white]{11};
\node (12) at (-3,1)[rectangle,gray!30,fill=white]{12};
\node (4') at (1,1){$\over[4]$};
\node (5') at (1,-1){$\over[5]$};
\node (6') at (2,0){$\over[6]$};
\node (7') at (3,-1)[rectangle]{$\over[7]$};
\node (8') at (2,-2)[rectangle]{$\over[8]$};
\node (9') at (1,-3)[rectangle]{$\over[9]$};
\node (12') at (1,3)[rectangle,gray!30,fill=white]{$\over[12]$};
\node (11') at (2,2)[rectangle,gray!30,fill=white]{$\over[11]$};
\node (10') at (3,1)[rectangle,gray!30,fill=white]{$\over[10]$};
\node at (-1.5,-4)[draw=none]{$Q_{A_3}^{\bi}$};
\node at (1.5,-4)[draw=none]{$Q_{A_3}^{\over[\bi]}$};
\drawpath{1,3,6,1',4',6',7'}{red, thick}
\drawpath{2,5,2',5',8'}{red, thick}
\drawpath{4,3',9'}{red, thick}
\drawpath{8',6',5',4',2',6,5,3,2}{thick}
\drawpath{9',5',3',5,4}{thick}
\drawpath{7',10',6',11',4',12',1',10,6,11,3,12,1}{thick, gray, opacity=0.3}
\drawpath{4,2,1}{dashed}
\drawpath{7',8',9'}{dashed}
\drawpath{12,11,10}{dashed,gray, opacity=0.3}
\drawpath{12',11',10'}{dashed,gray, opacity=0.3}
\end{tikzpicture}
$\simeq_{\mu_{\cR_1}}$
\begin{tikzpicture}[baseline=(3), every node/.style={inner sep=0, minimum size=0.4cm, thick, fill=white, draw,circle}, x=0.8cm, y=0.8cm]
\node (1) at (-3,-1)[rectangle]{$1$};
\node (2) at (-2,-2)[rectangle]{$2$};
\node (3) at (-2,0){$3$};
\node (4) at (-1,-3)[rectangle]{$4$};
\node (5) at (-1,-1){$5$};
\node (6) at (-1,1){$6$};
\node (1') at (0,2){$\over[1]$};
\node (2') at (0,0){$\over[2]$};
\node (3') at (0,-2){$\over[3]$};
\node (10) at (-1,3)[rectangle,gray!30,fill=white]{10};
\node (11) at (-2,2)[rectangle,gray!30,fill=white]{11};
\node (12) at (-3,1)[rectangle,gray!30,fill=white]{12};
\node (4') at (1,1){$\over[4]$};
\node (5') at (1,-1){$\over[5]$};
\node (6') at (2,0){$\over[6]$};
\node (7') at (3,-1)[rectangle]{$\over[7]$};
\node (8') at (2,-2)[rectangle]{$\over[8]$};
\node (9') at (1,-3)[rectangle]{$\over[9]$};
\node (12') at (1,3)[rectangle,gray!30,fill=white]{$\over[12]$};
\node (11') at (2,2)[rectangle,gray!30,fill=white]{$\over[11]$};
\node (10') at (3,1)[rectangle,gray!30,fill=white]{$\over[10]$};
\drawpath{1,5,5',7'}{red, thick}
\drawpath{2,3',8'}{red, thick}
\drawpath{4,9'}{red, thick}
\drawpath{8',5',3',5,2}{thick}
\drawpath{7',6',5',2',5,3,1}{thick}
\drawpath{3,6,1',4',6,2',4',6'}{thick}
\drawpath{6',10',4',11',1',12',10,1',11,6,12,3}{gray!30,thick}
\drawpath{9',3',4}{thick}
\drawpath{4,2,1}{dashed}
\drawpath{7',8',9'}{dashed}
\drawpath{12,11,10}{dashed,gray, opacity=0.3}
\drawpath{12',11',10'}{dashed,gray, opacity=0.3}
\node at (4,-1.5)[draw=none]{$Q_{A_3}^{\bi}$};
\node at (4,1.5)[draw=none]{$Q_{A_3}^{\over[\bi]}$};
\end{tikzpicture}
\caption{$\cP_\l^\fb\simeq \til{\cP^\fb}$ in type $A_3$ drawn symmetrically, with $F_i$-paths shown in red}\label{mutateA3}
\end{figure}

In general for the current choice of $\bi$, the $E_i$-paths are the unique shortest path that goes from $\cF_{out}$ to $\cE$ corresponding to the root index $i$. In the case of type $A_3$ they are given by (cf. Figure \ref{A3q})
\Eqn{
\cP_{E_1}&=(7,10), \\
\cP_{E_2}&=(8,6,11),\\
\cP_{E_3}&=(9,5,3,12).
}
Hence by \eqref{guv} we easily get
\Eqn{
g_b(\til{\be}_1^-)&=g_b(X_7),\\
g_b(\til{\be}_2^-)&=g_b(X_{6,8})g_b(X_8), \\
g_b(\til{\be}_3^-)&=g_b(X_{3,5,9})g_b(X_{5,9})g_b(X_9).
}
The $F_i$-paths are given by the shortest path joining $\cF_{in}$ to $\cF_{out}$
\Eqn{
\cP_{F_1}&=(1,3,6,7), \\
 \cP_{F_2}&=(2,5,8), \\
 \cP_{F_3}&=(4,9).
}

Identifying $X_{\over[1]}:=X_7\ox X_{\over[1]}, X_{\over[2]}:=X_8\ox X_{\over[2]}$ and $X_{\over[3]}:=X_9\ox X_{\over[3]}$, we obtain using \eqref{phi1} the unitary transformation $\Phi_1$ corresponding to a sequence of 10 mutations:
\Eqn{
\Phi_1&=g_b(X_{\over[1],\over[4],\over[6]})g_b(X_{6,\over[2],\over[5]})g_b(X_{\over[2],\over[5]})g_b(X_{\over[1],\over[4]})g_b(X_{3,5,\over[3]})g_b(X_{5,\over[3]})g_b(X_{\over[3]})g_b(X_{6,\over[2]})g_b(X_{\over[2]})g_b(X_{\over[1]}).
}
According to \cite{Ip16}, it corresponds to mutation at $$\mu_{\cR_1}=(\over[1],\over[2],6,\over[3],5,3,\over[4],\over[5],\over[2],\over[6]).$$

The mutation sequence for $Q^{\bi\bi}$ and the unitary transform $\Phi_3$ is identical up to re-indexing since $Q^{\bi}\simeq Q^{\over[\bi]}$. Hence together with the shifting as in Lemma \ref{indep} gives $\til{\cP^\fb}\ox \til{\cP^\fb}\simeq \til{\cP^\fb}\ox \cM$ as required. 

We can further write down the transformations $\Phi_k$ explicitly as unitary operators on the Hilbert space $L^2(\R^N)$ easily. For example, using the notation in Theorem \ref{FKaction}, the representation $\til{\cP^\fb}\simeq L^2(\R^6)$ for $\bi=(1,2,1,3,2,1)$ is given by
\Eqn{
\bf_1&=e^{\pi b(-2u_1+u_2-2u_3+u_5-u_6+2p_6)}+e^{\pi b(-2u_1+u_2-u_3+2p_3)}+e^{\pi b(-u_1+2p_1)},\\
\bf_2&=e^{\pi b(u_1-2u_2+u_3+u_4-u_5+2p_5)}+e^{\pi b(u_1-u_2+2p_2)},\\
\bf_3&=e^{\pi b(u_2-u_4+2p_4)},\\
K_1'&=e^{\pi b(2u_1-u_2+2u_3-u_5+2u_6)},\\
K_2'&=e^{\pi b(-u_1+2u_2-u_3-u_4+2u_5-u_6)},\\
K_3'&=e^{\pi b(-u_2+2u_4-u_5)}.
}
Then using $'$ for the second component, and $$\D(\bf_i)=\bf\ox 1+K_i' \ox \bf_i,$$ we obtain the unitary equivalence $\til{\cP^\fb}\ox \til{\cP^\fb}\simeq \til{\cP^\fb}\ox \cM$ where
\Eqn{
Ad_{\Phi_3}\cdot \D(\bf_i) &= \bf_i\ox 1,\\
Ad_{\Phi_3}\cdot \D(K_i) &= K_i\ox K_i,
}
is given by
\Eqn{\Phi_3&=g_b(e^{\pi b(3u_1-u_2+2u_3-u_5+2u_6-u_1'-2p_1+2p_1')})g_b(e^{\pi b(-5u_1+2u_2-u_4+2u_5-u_6+u_1'-u_2'+2p_1-2p_2-2p_3+2p_2')})\\
&g_b(e^{\pi b(-2u_1+3u_2-u_3-u_4+2u_5-u_6+u_1'-u_2'-2p_2+2p_2')})g_b(e^{\pi b(3u_1-u_2+2u_3-u_5+2u_6+2u_1'+u_2'-u_3'-2p_1+2p_3')})\\
&g_b(e^{\pi b(4u_1-u_2+2u_4-u_5+u_6+u_2'-u_4'+2p_2+2p_3-2p_4-2p_5-2p_6+2p_4')})g_b(e^{\pi b(-u_2-u_3+2u_4+u_2'-u_4'+2p_2-2p_4-2p_5+2p_4')})\\
&g_b(e^{\pi b(-2u_2+3u_4-u_5+u_2'-u_4'-2p_4+2p_4')})g_b(e^{\pi b(-5u_1+2u_2-u_4+2u_5-u_6+u_1'-2u_2'+u_3'+u_4'-u_5'+2p_1-2p_2-2p_3+2p_5')})\\
&g_b(e^{\pi b(-2u_1+3u_2-u_3-u_4+2u_5-u_6+u_1'-2u_2'+u_3'+u_4'-u_5'-2p_2+2p_5')})\\
&g_b(e^{\pi b(3u_1-u_2+2u_3-u_5+2u_6-2u_1'+u_2'-2u_3'+u_5'-u_6'-2p_1+2p_6')}),
}
while the unitary transformation $\cS$ in Proposition \ref{K1}:
\Eqn{
\cS\cdot (K_i\ox K_i)&=K_i\ox 1,\\
\cS\cdot (\bf_i\ox 1) &= \bf_i\ox 1
}
is given by the shifts
\Eqn{\cS&=(2p_1\mapsto 2p_1-u_1'-u_3'-u_6')\circ(2p_2\mapsto 2p_2+u_1'-u_2'+u_3'-u_5'+u_6')\circ\\
&(2p_3\mapsto 2p_3+2u_1'+u_2'+2u_3'+u_5'+2u_6')\circ(2p_4\mapsto 2p_4+u_2'-u_4'+u_5')\circ \\
&(2p_5\mapsto 2p_5+u_1'-2u_2'+u_3'+u_4'-2u_5'+u_6')\circ(2p_6\mapsto 2p_6-2u_1'+u_2'-2u_3'+u_5'-2u_6')\circ \\
&(u_1\mapsto u_1-u_1'-u_3'-u_6')\circ(u_2\mapsto u_2-u_2'-u_5')\circ(u_4\mapsto u_4-u_4'),
}
where $$2p_k\mapsto 2p_k-\l$$ is realized by multiplication by the unitary function $e^{\pi i \l u_k}$ on $L^2(\R^6)$.

\subsection{Type $B_3$}\label{sec:b3}
Finally let us consider the situation where $Q^{\bi}\not\simeq Q^{\over[\bi]}$, hence we need to describe $\Phi_1$ and $\Phi_3$ separately. Let $$\bi=(1,2,1,2,3,2,1,2,3)$$ be the reduced word chosen in \cite{Ip16}, and $$\over[\bi]=(3,2,1,2,3,2,1,2,1)$$ be the reversed word. The basic quiver for $B_3$ is drawn in Figure \ref{B3q}. The positive representation $\cP_\l^\fb$ is given on $\cX^{\bi\over[\bi]}$ indexed as in the top of Figure \ref{mutateB3} by
\Eqn{
\bf_1&=X(1,3,6,\over[1],\over[4],\over[6]),& K_1' &=X_{1,3,6,\over[1],\over[4],\over[6],\over[7]},\\
\bf_2&=X(2,5,\over[2],\over[5]),&K_2' &=X_{2,5,\over[2],\over[5],\over[8]},\\
\bf_3&=X(4,\over[3]),& K_3'& =X_{4,\over[3],\over[9]}.
}

\begin{figure}[H]
\centering
\begin{tikzpicture}[every node/.style={inner sep=0, minimum size=0.5cm, thick, draw,circle}, x=0.5cm, y=0.75cm]
\node (1) at (0,4)[rectangle]{$1$};
\node (2) at (0,2)[rectangle]{$2$};
\node (3) at (3,4){$3$};
\node (4) at (3,2){$4$};
\node (5) at (0,0)[rectangle]{$5$};
\node (6) at (6,2){$6$};
\node (7) at (9,4){$7$};
\node (8) at (9,2){$8$};
\node (9) at (6,0){$9$};
\node (13) at (3,6)[rectangle, gray!30]{$13$};
\node (14) at (6,6)[rectangle, gray!30]{$14$};
\node (15) at (9,6)[rectangle, gray!30]{$15$};
\node (1') at (12,0){$\over[1]$};
\node (2') at (12,2){$\over[2]$};
\node (3') at (12,4){$\over[3]$};
\node (4') at (15,2){$\over[4]$};
\node (5') at (18,0){$\over[5]$};
\node (6') at (18,2){$\over[6]$};
\node (7') at (15,4){$\over[7]$};
\node (8') at (21,2){$\over[8]$};
\node (9') at (21,4){$\over[9]$};
\node (10') at (24,4)[rectangle]{$\over[10]$};
\node (11') at (24,2)[rectangle]{$\over[11]$};
\node (12') at (24,0)[rectangle]{$\over[12]$};
\node (15') at (15,6)[rectangle, gray!30]{$\over[15]$};
\node (14') at (18,6)[rectangle, gray!30]{$\over[14]$};
\node (13') at (21,6)[rectangle, gray!30]{$\over[13]$};
\drawpath{3,13,1}{thin, gray!30}
\drawpath{6,14,4}{vthick, gray!30}
\drawpath{1',15,9}{vthick, gray!30}
\drawpath{10',13',9'}{thin, gray!30}
\drawpath{8',14',6'}{vthick, gray!30}
\drawpath{5',15',1'}{vthick, gray!30}
\drawpath{1,3,7,3',7',9',10'}{red, thin}
\drawpath{2,4,6,8,2',4',6',8',11'}{red,vthick}
\drawpath{5,9,1',5',12'}{red,vthick}
\drawpath{12',6',5',2',9,6,5}{vthick}
\drawpath{11',9',8',7',4',3',8,7,4,3,2}{vthick}
\drawpath{5,2,1}{dashed, vthick}
\drawpath{10',11',12'}{dashed,vthick}
\drawpath{13,14,15}{dashed,vthick,gray!30}
\drawpath{15',14',13'}{dashed,vthick,gray!30}
\node at (6,-1)[draw=none]{$Q_{B_3}^{\bi}$};
\node at (18,-1)[draw=none]{$Q_{B_3}^{\over[\bi]}$};
\end{tikzpicture}

\begin{tikzpicture}[baseline=(1), every node/.style={inner sep=0, minimum size=0.5cm, thick, draw,circle}, x=0.5cm, y=0.75cm]
\node (1) at (0,4)[rectangle]{$1$};
\node (2) at (0,2)[rectangle]{$2$};
\node (3) at (3,10.5){$3$};
\node (4) at (3,2){$4$};
\node (5) at (0,0)[rectangle]{$5$};
\node (6) at (6,10.5){$6$};
\node (7) at (3,7.5){$7$};
\node (8) at (9,2){$8$};
\node (9) at (6,9){$9$};
\node (13) at (3,12)[rectangle,gray!30]{$13$};
\node (14) at (6,12)[rectangle,gray!30]{$14$};
\node (15) at (9,12)[rectangle,gray!30]{$15$};
\node (1') at (9,9){$\over[1]$};
\node (2') at (6,7.5){$\over[2]$};
\node (3') at (3,6){$\over[3]$};
\node (4') at (6,6){$\over[4]$};
\node (5') at (9,6){$\over[5]$};
\node (6') at (6,2){$\over[6]$};
\node (7') at (3,4){$\over[7]$};
\node (8') at (6,0){$\over[8]$};
\node (9') at (9,4){$\over[9]$};
\node (10') at (12,4)[rectangle]{$\over[10]$};
\node (11') at (12,2)[rectangle]{$\over[11]$};
\node (12') at (12,0)[rectangle]{$\over[12]$};
\node (13') at (0,7.5)[rectangle,gray!30]{$\over[13]$};
\node (14') at (0,9)[rectangle,gray!30]{$\over[14]$};
\node (15') at (0,10.5)[rectangle,gray!30]{$\over[15]$};
\drawpath{13,3}{thin, gray!30}
\drawpath{15,1',14,6,13}{vthick, gray!30}
\drawpath{3',13',7}{thin, gray!30}
\drawpath{2',14',9}{vthick, gray!30}
\drawpath{1',15',15}{vthick, gray!30}
\drawpath{15,14,13}{vthick, dashed, gray!30}
\drawpath{15',14',13'}{vthick, dashed, gray!30}
\drawpath{11',8',6',5}{vthick}
\drawpath{10',8,9',4,7',2}{vthick}
\drawpath{7',3',1}{thin}
\drawpath{6',4',4}{vthick}
\drawpath{12',5',8'}{vthick}
\drawpath{3',4',5'}{vthick}
\drawpath{3,7,3'}{thin}
\drawpath{3,6,9,1',5',9,2',4',7,2',3}{vthick}
\drawpath{1,7',9',10'}{red,thin}
\drawpath{2,4,6',8,11'}{red,vthick}
\drawpath{5,8',12'}{red,vthick}
\drawpath{5,2,1}{dashed, vthick}
\drawpath{12',11',10'}{dashed,vthick}
\node at (-2,6)[draw=none]{$\simeq_{\mu_{\cR_1}}$};
\node at (14,2)[draw=none]{$Q_{B_3}^{\bi}$};
\node at (14,9)[draw=none]{$Q_{B_3}^{\over[\bi]}$};
\end{tikzpicture}
\caption{$\cP_\l^\fb\simeq \til{\cP^\fb}$ in type $B_3$ with $F_i$ paths shown in red.}\label{mutateB3}
\end{figure}

The $E_i$-paths are described explicitly in \cite{Ip16} given by (cf Figure \ref{B3q})
\Eqn{
\cP_{E_1}&=(10,8,7,4,3,13), \\
\cP_{E_2}&=(11,9,6,14), \\
 \cP_{E_3}&=(12,15),
 }
and correspond to the following decomposition of $g_b$ which follows from \eqref{guv} and \eqref{gdouble}:
\Eqn{
g_b({\be'}_1^-)&=g_{b_s}(X_{3,4,7,8,10})g_{b_s}(X_{4,7,8,10})g_b(X_{4,7^2,8^2,10^2})g_{b_s}(X_{7,8,10})g_{b_s}(X_{8,10})g_b(X_{8,10^2})g_{b_s}(X_{10}),\\
g_b({\be'}_2^-)&=g_b(X_{6,9,11})g_b(X_{9,11})g_b(X_{11}),\\
g_b({\be'}_3^-)&=g_b(X_{12}).
}

The $F_i$-paths are given by the horizontal red paths:
\Eqn{
\cP_{F_1}&=(1,3,7,10), \\
 \cP_{F_2}&=(2,4,6,8,11), \\
  \cP_{F_3}&=(5,9,12).
}
Identifying $X_{\over[1]}:=X_{12}\ox X_{\over[1]}, X_{\over[2]}:=X_{11}\ox X_{\over[2]}$ and $X_{\over[3]}:=X_{10}\ox X_{\over[3]}$ in $\cX^{\bi\over[\bi]}$, we obtain the unitary transformation $\Phi_1$ corresponding to a sequence of 35 mutations as follows:
{\Eqn{
\Phi_1=&g_{b_s}(X_{3,4,7,8,\over[3],\over[7],\over[9]})g_{b_s}(X_{4,7,8,\over[3],\over[7],\over[9]})g_b(X_{4,7^2,8^2,\over[3]^2,\over[7]^2,\over[9]^2})g_{b_s}(X_{7,8,\over[3],\over[7],\over[9]})g_{b_s}(X_{8,\over[3],\over[7],\over[9]})g_b(X_{8,\over[3]^2,\over[7]^2,\over[9]^2})\\
&g_{b_s}(X_{\over[3],\over[7],\over[9]})g_b(X_{6,9,\over[2],\over[4],\over[6],\over[8]})g_b(X_{9,\over[2],\over[4],\over[6],\over[8]})g_b(X_{\over[2],\over[4],\over[6],\over[8]})g_{b_s}(X_{3,4,7,8,\over[3],\over[7]})g_{b_s}(X_{4,7,8,\over[3],\over[7]})g_b(X_{4,7^2,8^2,\over[3]^2,\over[7]^2})\\
&g_{b_s}(X_{7,8,\over[3],\over[7]})g_{b_s}(X_{8,\over[3],\over[7]})g_b(X_{8,\over[3]^2,\over[7]^2})g_{b_s}(X_{\over[3],\over[7]})g_b(X_{6,9,\over[2],\over[4],\over[6]})g_b(X_{9,\over[2],\over[4],\over[6]})g_b(X_{\over[2],\over[4],\over[6]})g_b(X_{\over[1],\over[5]})g_b(X_{6,9,\over[2],\over[4]})\\
&g_b(X_{9,\over[2],\over[4]})g_b(X_{\over[2],\over[4]})g_{b_s}(X_{3,4,7,8,\over[3]})g_{b_s}(X_{4,7,8,\over[3]})g_b(X_{4,7^2,8^2,\over[3]^2})g_{b_s}(X_{7,8,\over[3]})g_{b_s}(X_{8,\over[3]})g_b(X_{8,\over[3]^2})g_{b_s}(X_{\over[3]})\\
&g_b(X_{6,9,\over[2]})g_b(X_{9,\over[2]})g_b(X_{\over[2]})g_b(X_{\over[1]}),
}
which corresponds to the mutation sequence of $Q^{\bi\over[\bi]}$ at $$\mu=(\over[1],\over[2],9,6,\over[3],8,\over[3],7,4,7,3,\over[4],\over[2],9,\over[5],\over[6],\over[4],\over[2],\over[7],8,\over[7],\over[3],4,\over[3],7,\over[8],\over[6],\over[4],\over[9],8,\over[9],\over[7],4,\over[7],\over[3]).$$

Finally, the unitary transformation $\Phi_3$ is constructed similarly, where we use $\bf^{k,+}$ instead of $\bf^{k,-}$, and $\bi$ instead of $\over[\bi]$. Using the standard indexing as in Figure \ref{finalB3}, and identifying $X_{1'}:=X_{10}\ox X_{1'}, X_{2'}:=X_{11}\ox X_{2'}$ and $X_{3'}:=X_{12}\ox X_{3'}$ in $\cX^{\bi\bi}$ we have
{
\Eqn{
\Phi_3=&g_b(X_{5',9'})g_b(X_{6,9,2',4',6',8'})g_b(X_{9,2',4',6',8'})g_b(X_{2',4',6',8'})g_{b_s}(X_{3,4,7,8,1',3',7'})g_{b_s}(X_{4,7,8,1',3',7'})\\
&g_b(X_{4,7^2,8^2,{1'}^2,{3'}^2,{7'}^2})g_{b_s}(X_{7,8,1',3',7'})g_{b_s}(X_{8,1',3',7'})g_b(X_{8,{1'}^2,{3'}^2,{7'}^2})g_{b_s}(X_{1',3',7'})g_b(X_{6,9,2',4',6'})\\
&g_b(X_{9,2',4',6'})g_b(X_{2',4',6'})g_b(X_{5'})g_b(X_{6,9,2',4'})g_b(X_{9,2',4'})g_b(X_{2',4'})g_{b_s}(X_{3,4,7,8,1',3'})g_{b_s}(X_{4,7,8,1',3'})\\
&g_b(X_{4,7^2,8^2,{1'}^2,{3'}^2})g_{b_s}(X_{7,8,1',3'})g_{b_s}(X_{8,1',3'})g_b(X_{8,{1'}^2,{3'}^2})g_{b_s}(X_{1',3'})g_b(X_{6,9,2'})g_b(X_{9,2'})g_b(X_{2'})\\
&g_{b_s}(X_{3,4,7,8,1'})g_{b_s}(X_{4,7,8,1'})g_b(X_{4,7^2,8^2,{1'}^2})g_{b_s}(X_{7,8,1'})g_{b_s}(X_{8,1'})g_b(X_{8,{1'}^2})g_{b_s}(X_{1'}),
}
which corresponds to the mutation sequence of $Q^{\bi\bi}$ at
$$\mu=(1',8,1',7,4,7,3,2',9,6,3',8,3',1',4,1',7,4',2',9,5',6',4',2',7',8,7',3',4,3',1',8',6',4',9').$$

\begin{figure}[H]
\centering
\begin{tikzpicture}[every node/.style={inner sep=0, minimum size=0.5cm, thick, draw,circle}, x=0.5cm, y=0.75cm]
\node (1) at (0,4)[rectangle]{$1$};
\node (2) at (0,2)[rectangle]{$2$};
\node (3) at (3,4){$3$};
\node (4) at (3,2){$4$};
\node (5) at (0,0)[rectangle]{$5$};
\node (6) at (6,2){$6$};
\node (7) at (9,4){$7$};
\node (8) at (9,2){$8$};
\node (9) at (6,0){$9$};
\node (13) at (3,6)[rectangle, gray!30]{$13$};
\node (14) at (6,6)[rectangle, gray!30]{$14$};
\node (15) at (9,6)[rectangle, gray!30]{$15$};
\node (1') at (12,4){$1'$};
\node (2') at (12,2){$2'$};
\node (3') at (15,4){$3'$};
\node (4') at (15,2){$4'$};
\node (5') at (12,0){$5'$};
\node (6') at (18,2){$6'$};
\node (7') at (21,4){$7'$};
\node (8') at (21,2){$8'$};
\node (9') at (18,0){$9'$};
\node (10') at (24, 4)[rectangle]{$10'$};
\node (11') at (24, 2)[rectangle]{$11'$};
\node (12') at (24, 0)[rectangle]{$12'$};
\node (13') at (15,6)[rectangle, gray!30]{$13'$};
\node (14') at (18,6)[rectangle, gray!30]{$14'$};
\node (15') at (21,6)[rectangle, gray!30]{$15'$};
\drawpath{3,13,1}{thin, gray!30}
\drawpath{6,14,4}{vthick, gray!30}
\drawpath{5',15,9}{vthick, gray!30}
\drawpath{3',13',1'}{thin, gray!30}
\drawpath{6',14',4'}{vthick, gray!30}
\drawpath{12',15',9'}{vthick, gray!30}
\drawpath{1,3,7,1',3',7',10'}{red, thin}
\drawpath{2,4,6,8,2',4',6',8',11'}{red,vthick}
\drawpath{5,9,5',9',12'}{red,vthick}
\drawpath{10',8',7',4',3',2',1',8,7,4,3,2}{vthick}
\drawpath{11',9',6',5',2',9,6,5}{vthick}
\drawpath{5,2,1}{dashed, vthick}
\drawpath{12',11',10'}{dashed,vthick}
\drawpath{13,14,15}{dashed,vthick,gray!30}
\drawpath{13',14',15'}{dashed,vthick,gray!30}
\node at (6,-1)[draw=none]{$Q_{B_3}^{\bi}$};
\node at (18,-1)[draw=none]{$Q_{B_3}^{\bi}$};
\end{tikzpicture}

\begin{tikzpicture}[baseline=(1), every node/.style={inner sep=0, minimum size=0.5cm, thick, draw,circle}, x=0.5cm, y=0.75cm]
\node (1) at (0,4)[rectangle]{$1$};
\node (2) at (0,2)[rectangle]{$2$};
\node (3) at (3,10.5){$3$};
\node (4) at (3,2){$4$};
\node (5) at (0,0)[rectangle]{$5$};
\node (6) at (6,10.5){$6$};
\node (7) at (3,7.5){$7$};
\node (8) at (9,2){$8$};
\node (9) at (6,9){$9$};
\node (13) at (3,12)[rectangle,gray!30]{$13$};
\node (14) at (6,12)[rectangle,gray!30]{$14$};
\node (15) at (9,12)[rectangle,gray!30]{$15$};
\node (1') at (9,9){$1'$};
\node (2') at (6,7.5){$2'$};
\node (3') at (3,6){$3'$};
\node (4') at (6,6){$4'$};
\node (5') at (9,6){$5'$};
\node (6') at (6,2){$6'$};
\node (7') at (3,4){$7'$};
\node (8') at (6,0){$8'$};
\node (9') at (9,4){$9'$};
\node (10') at (12,4)[rectangle]{$10'$};
\node (11') at (12,2)[rectangle]{$11'$};
\node (12') at (12,0)[rectangle]{$12'$};
\node (13') at (0,10.5)[rectangle,gray!30]{$13'$};
\node (14') at (0,9)[rectangle,gray!30]{$14'$};
\node (15') at (0,7.5)[rectangle,gray!30]{$15'$};
\drawpath{13,3,13',13}{thin, gray!30}
\drawpath{15,14,13}{vthick, dashed, gray!30}
\drawpath{13',14',15'}{vthick, dashed, gray!30}
\drawpath{9,14',6}{vthick, gray!30}
\drawpath{5',15',1'}{vthick, gray!30}
\drawpath{9,15,1'}{vthick, gray!30}
\drawpath{3,14,6}{vthick, gray!30}
\drawpath{11',8',6',5}{vthick}
\drawpath{10',8,9',4,7',2}{vthick}
\drawpath{7',3',1}{thin}
\drawpath{6',4',4}{vthick}
\drawpath{12',5',8'}{vthick}
\drawpath{3,7,3'}{thin}
\drawpath{3',2',7,6,3}{vthick}
\drawpath{6,9,2',4'}{vthick}
\drawpath{4',1',9}{vthick}
\drawpath{1',5'}{vthick}
\drawpath{1,7',9',10'}{red,thin}
\drawpath{2,4,6',8,11'}{red,vthick}
\drawpath{5,8',12'}{red,vthick}
\drawpath{5,2,1}{dashed, vthick}
\drawpath{12',11',10'}{dashed,vthick}
\node at (-2,6)[draw=none]{$\simeq_{\mu_{\cR_3}}$};
\node at (14,2)[draw=none]{$Q_{B_3}^{\bi}$};
\node at (14,9)[draw=none]{$Q_{B_3}^{\bi}$};
\end{tikzpicture}
\caption{$\til{\cP^\fb}\ox \til{\cP^\fb}\simeq \til{\cP^\fb}\ox \cM$ in type $B_3$ with $F_i$ paths shown in red.}\label{finalB3}
\end{figure}

\end{document}